\documentclass{amsart}

\usepackage{epsfig} 
\usepackage{times} 
\usepackage{fixmath}
\usepackage{longtable}
\usepackage{booktabs}
\usepackage{mathtools}
\usepackage{verbatim}

\newtheorem{theorem}{Theorem}[section]
\newtheorem{lemma}[theorem]{Lemma}

\newtheorem{proposition}[theorem]{Proposition}

\theoremstyle{definition}

\newtheorem{example}[theorem]{Example}

\theoremstyle{remark}
\newtheorem{remark}[theorem]{Remark}

\def\cal{\mathcal}

\def\citep{\cite}
\def\p{{\partial}}
\def\maC{\mathcal C}

\def\cki{\iota}

\numberwithin{equation}{section}

\begin{document}

\title{Graphon Mean Field Games and the GMFG Equations}

\author{Peter E. Caines}
\address{Department
of Electrical and Computer Engineering, McGill University, Montreal, QC, Canada}
\email{peterc@cim.mcgill.ca}
\thanks{This work was supported by NSERC and AFOSR (P. E. Caines) and NSERC (M. Huang).}


\author{Minyi Huang}
\address{School  of Mathematics and Statistics, Carleton University,
Ottawa, ON, Canada}
\email{mhuang@math.carleton.ca}


\date{Aug 24, 2020; revised Apr 12, 2021, Jun 15, 2021, Dec 28, 2021.}



\begin{abstract}
 The emergence of the graphon theory of large networks and their infinite limits has enabled the formulation of a theory of the centralized control of dynamical systems distributed on asymptotically infinite networks
\cite{ShuangPeterCDC17,ShuangPeterCDC18}.
 Furthermore, the study of the decentralized  control of such systems was initiated
 in \cite{CainesHuangCDC2018,CH19},   where  Graphon Mean Field Games (GMFG)  and the GMFG equations were formulated for the analysis of non-cooperative dynamic games  on  unbounded networks. In that work, existence and uniqueness results were  introduced for the GMFG equations, together with an
$\epsilon$-Nash theory for GMFG systems which relates infinite population equilibria on infinite networks to finite population equilibria on finite networks. Those results are rigorously established in this paper.
\end{abstract}

\maketitle

\section{Introduction}

One response  to the problems arising in the analysis of systems of great complexity is to pass to an appropriately formulated infinite limit. This approach has a distinguished history since it is the conceptual principle underlying the celebrated Boltzmann Equation of statistical mechanics and that of the fundamental Navier-Stokes equation of fluid mechanics (see e.g. \citep{pauli2000thermodynamics, herron2008partial,gallagher2019newtonnavierstokes,gallagher2013newton}). Similarly  the Fokker-Planck-Kolmogorov (FPK) equation  for the macroscopic flow of probabilities \citep{doob1953stochastic, karatzas2012brownian} is used to describe a vast range of phenomena which at a micro or mezzo level are modelled via the random interactions of discrete entities.

 The work in this paper is formulated within two recent theories which were  developed with an analogous motive to that  above, namely
 Mean Field Game (MFG) theory for the analysis of equilibria in very large populations of non-cooperative agents (see \citep{HMC06, HCM07, LL06a, LL06b, CarmonaDelarueBook2018a, CarmonaDelarueBook2018b, CHM17}), and the graphon theory of the infinite limits of graphs and networks  (see
\cite{lovasz2006limits,borgs2006graph,borgs2008convergent,borgs2012convergent, lovasz2012large}).

A mathematically rigorous study of MFG systems with state values in finite graphs is provided in \citep{G15}, and  MFG systems where the agent subsystems are defined at the nodes (vertices) of finite random Erd\"os-R\'enyi graphs are treated in \citep{delarue2017mean}. The system behaviour in \cite{G15} is  subject to a fixed underlying network. The random graphs in \cite{delarue2017mean} have unbounded growth but do not create spatial distinction of the agents due to  symmetry properties of the interactions. However, graphon  theory gives a rigorous formulation of the notion of limits for infinite sequences of networks of increasing size,  and the first application of graphon theory in dynamics appears to be in the  work of Medvedev
 \citep{medvedev2014nonlineardense,medvedev2014nonlinear}, and
  Kaliuzhnyi-Verbovetskyi and Medvedev \citep{kaliuzhnyi2017semilinear}.
  The law of large numbers for graphon mean field systems is proven in \cite{BCW20} as a generalization of results for standard interacting particle systems.
Furthermore, the work in \cite{QT15} derives the McKean-Vlasov limit for a network of agents described by delay stochastic differential equations  that are coupled by randomly generated connections.

The first applications of graphon theory in systems  and control theory are those in \citep{ShuangPeterNetSci17, ShuangPeterSIAM17, ShuangPeterCDC17, ShuangPeterCDC18,GC20} which treat the centralized and distributed control of arbitrarily large networks of linear dynamical control systems for which a direct solution would be  intractable. Approximate control is achieved by solving control problems on the infinite  limit graphon and then applying control laws derived from those solutions on the finite network of interest. The analogy with the strategies for finding feedback laws resulting in  $\epsilon$-Nash equilibria in the MFG framework is obvious. In this connection we note that  work on static game theoretic equilibria for infinite populations on graphons  was reported in {\cite{PariseOzdagler2018}}.

 A natural framework for the formulation of game theoretic  problems
  involving large populations of  agents distributed over  large networks is given by Mean Field Game theory defined on graphons. The resulting basic idea and the associated fundamental equations for what we term  Graphon  Mean Field Game  (GMFG) systems  and the GMFG equations are the subject of
  the current paper and its predecessors \cite{CainesHuangCDC2018,CH19}. The GMFG  equations are of significant  generality since they permit the study, in the limit, of both dense and sparse, infinite networks of non-cooperative
 dynamical agents. Moreover the classical  MFG equations are retrieved
 as a special case. We observe that an early analysis of linear quadratic (LQ) models in mean field games on networks with non-uniform edge weightings can be found in  \cite{HCM10}.
However, in that work there was no application of graphon theory, and in the uniform system parameter case there is one agent per node and a single mean field,  whereas in the present work there is a subpopulation with its own mean field at each node.

The basic $\epsilon$-Nash equilibrium result in MFG theory and its corresponding  form in GMFG theory are vital for the application of MFG derived control laws. This is the case since the solution of the MFG and GMFG  equations is necessarily simpler than the effectively intractable task of finding the solution to  the game problems for the large finite population systems. Indeed, this was one of the original motives for the creation of MFG theory and it is a basic feature of  graphon systems control theory \cite{ShuangPeterNetSci17}.

The paper is organized as follows.
Section \ref{sec:graphon} provides preliminary materials on graphons. Section \ref{sec:gmfg} introduces the GMFG  equation system  and proves the existence and uniqueness of a solution. For the decentralized strategies determined by the GMFG equations, an $\epsilon$-Nash equilibrium theorem is proven in Section \ref{sec:enash}.  The GMFG equations are illustrated by an LQ example in Section \ref{sec:lq}.

\begin{longtable}{rl}
\caption{Notation}\\
\toprule
$G_k$  & the $k$-th graph in a sequence of graphs\\
$g^k$  & weights of $G_k$ as a step function   \\
$M_k$  & the number of nodes in $G_k$ \\
${\mathcal C}_i$  & the cluster of agents residing at node $i$ of $G_k$ \\
${\mathcal C}(i)$ & the cluster that agent $i$ belongs to\\
$I_i^*$, $I^*(i)$ & the midpoint of an interval of length $1/M_k$\\
$g$ & the graphon function \\
$\mu_\alpha(t)$ &  the  local mean field generated by agents at vertex $\alpha\in [0,1]$\\
$\mu_G(t)$   & an ensemble of local mean fields $(\mu_\alpha(t))_{0\le \alpha \le 1}$ \\
${\mathcal M}_{[0,T]}  $  & a class of  $\mu_G(\cdot)$ satisfying a H\"older continuity condition\\
$C_T$ & the space of continuous functions on $[0,T]$\\
${\mathcal F}_T$ & $\sigma$-algebra induced by cylindrical sets in $C_T$
   \\
$(C_T, {\mathcal F}_T, m_\alpha)$ & probability measure space for the path space at vertex $\alpha $\\
$ {\bf M}_T$ & the set of probability measures on $(C_T, {\cal F}_T)$  \\
 $ D_T $&  Wasserstein metric on   $ {\bf M}_T $  \\
${\bf M}_T^G$ &  the product space $\prod_{\alpha\in [0,1]} {\bf M}_T$  \\
${\bf M}_T^{G0}$, ${\bf M}_T^{G1}$  &   subsets of ${\bf M}_T^G$   \\
$m_G$ & an ensemble of measures $(m_\alpha)_{0\le \alpha \le 1} \in {\bf M}_T^G$\\
${\rm Proj}_\alpha (m_G) $  & the component  $ m_\alpha $ at vertex $\alpha$\\
${\rm Marg}_t(m_\alpha) $ & the time $t$-marginal of $m_\alpha$\\
$x_\alpha$  & the state of a generic agent at
                             vertex $\alpha \in [0,1]$\\
$w_\alpha $ & a generic standard Brownian motion at vertex $\alpha$
\\
$\varphi(t,x_\alpha|\mu_G(\cdot);g_\alpha)$ & the best response at vertex $\alpha$  with $\mu_G(\cdot)$ given by the GMFG system;  \\
 & abbreviated as $\varphi(t,x_\alpha,g_\alpha)$ or $\varphi_\alpha$ \\
$\phi(t,x_\alpha|\mu_G(\cdot);g_\alpha)$ & the best response at vertex $\alpha$  with respect to an arbitrary $\mu_G(\cdot)$;\\
 & abbreviated as $\phi_\alpha(t,x_\alpha|\mu_G(\cdot))$ or $\phi_\alpha$\\
\bottomrule
\end{longtable}



\section{The Concept of a Graphon}
\label{sec:graphon}

The basic idea of the theory of graphons is that the edge structure of each finite cardinality network is represented by a step function density on the unit square in $\mathbb{R}^{2}$ on  which the  so-called cut norm and  cut metrics  are defined. The set of finite graphs endowed with the cut metric  then gives rise to a metric space, and the completion of this space is the space of  graphons. Let $\mathbf{G_0^{sp}}$ denote the linear space  of bounded symmetric  Lebesgue measurable functions $W: [0,1]^2 \rightarrow \mathbb{R}$, which are called kernels.
   The space $\mathbf{G^{sp}}$  of graphons is a subset of $\mathbf{G_0^{sp}}$ and consists of kernels  $W: [0,1]^2 \rightarrow [0,1]$ which
can be interpreted as weighted graphs on the vertex set $[0,1]$.
We note that functions $W\in \mathbf{G^{sp}}$ taking values in finite sets satisfy  this definition and so, in particular, graphons are  defined on
 finite graphs.

The cut norm of a kernel $W\in \mathbf{G_0^{sp}}$ then has the expression:
 \begin{equation}
 \| W \|_{\Box}=\sup_{M,T\subset [0,1]}\Big|\int_{M\times T}W(x,y)dxdy\Big| \nonumber
 \end{equation}
 with the supremum taking over all measurable subsets $M$ and $T$ of $[0,1]$.
 Denote the set of measure preserving bijections $[0, 1] \rightarrow [0, 1]$ by $S_{[0,1]}$.  The \emph{cut metric} between two graphons $V$ and $W$ is then given by
   $     \delta_{\Box}(W, V)=\inf_{\phi\in S_{[0,1]}}\|W^{\phi} -V\|_{\Box}$,
where $W^{\phi}(x,y) \coloneqq W(\phi(x),\phi(y))$ and any pair of graphons at zero distance  are identified  with each other.
The space  $(\mathbf{G^{sp}}, \delta_{\Box})$ is compact in the topology given by the cut metric \cite{lovasz2012large}. Furthermore, sets in $(\mathbf{G^{sp}}, \delta_{\Box})$
which are compact with respect to the $L^2$ metric are compact with respect to the cut metric.
 Since $\mathbf{G^{sp}}$ is compact in the cut metric all sequences of graphons have subsequential limits.

In this paper, we start with the modeling of the game of a finite population based on a finite graph.
Specifically, the population resides on a weighted finite graph $G_k$ with a set of nodes (or vertices) ${\cal V}_k=\{1, \ldots, M_k\}$ and weights $g^k_{ij}\in [0,1]$ for $(i,j)\in {\cal V}_k\times {\cal V}_k$, where a
 value   $g_{ii}^k$ is assigned in the case $i=j$.
We call $g_i^k \coloneqq(g_{i1}^k,\ldots,  g_{iM_k}^k)$ a section of $g^k$ at $i$.
Each node $l$ is occupied by a set of agents
which is called a cluster of the population and hence the number of clusters is $M_k$.
We  list the clusters as
$\maC_1, \ldots, \maC_{M_k}$.
Without loss of generality, we assume the $l$th cluster occupies node $l$.
Let $\maC(i)$ denote  the cluster that  agent $i$ belongs to. So $i\in \maC(i)$.
Our further analysis in the paper is based on the convergence of $g^k$ to a graphon limit $g$. We may naturally identify $(g^k_{ij})_{1\le i,j\le M_k}$ with a graphon $g^k(\alpha, \beta)$ as a step function defined on
$[0,1]\times[0,1]$ (see \cite{lovasz2012large}).
However,  convergence in the cut norm or  the cut metric is inadequate for the analysis in this paper as it does not capture sufficiently strong sectional information of the difference $g^k-g$.
We will adopt a different convergence notion strengthening the sectional  requirement as in assumption (H11) below.    To indicate its arguments, we may write $g(\alpha, \beta)$ or alternatively $g_{\alpha, \beta}$. We define the section of $g$ at $\alpha$ by $g_\alpha: \beta \mapsto g_{\alpha,\beta}$, $\beta\in[0,1]$.

Since clusters $\maC_{i_1}$ and $\maC_{i_2}$ reside on nodes $i_1$ and $i_2$ of $G_k$, respectively, we define $g^k_{\maC_{i_1}\maC_{i_2}}= g^k_{i_1 i_2}$.
Similarly, we define the section $g^k_{\maC_{i}}= g^k_{i}$.

We partition $[0,1]$ into $M_k$ subintervals of equal length.
Here $I_l^k=[(l-1)/M_k, l/M_k]$ for $1\le l\le M_k$.
When it is clear from the context, we omit the superscript $k$ and write $I_l$. To relate the clusters of agents to the vertex set $[0,1]$, we let the cluster $\maC_l$ correspond to $I_l$.

Throughout this paper, $C, C_0, C_1, \ldots$ denote generic constants,
which do not depend on the graph index $k$ and population size $N$ and  may vary from place to place.

\section{Graphon MFG Systems and the GMFG Equations}

\label{sec:gmfg}

 \subsection{The Standard MFG Model and Its Graphon Generalization}

 \label{agentdynamics}

 In the diffusion based models of large population games the state evolution of  a collection of $N$ agents  ${\cal{A}}_{i}, 1 \le i \le N <  \infty, $ is specified by a set  of $N$ controlled  stochastic differential equations (SDEs). A simplified form of the general case  is given by the following set of controlled SDEs
 which for each agent ${\cal{A}}_{i}$ includes state coupling with {\it all} other agents:
\begin{equation}\label{NLMinorDynamics}
dx_i(t)  = \frac{1}{N} \sum_{j=1}^N f(x_i(t),u_i(t), x_j(t))  dt + \sigma dw_i(t), \\
 \end{equation}
where  $x_i \in \mathbb{R}^n$ is the state, $u_i \in \mathbb{R}^{n_u}$  the control input, and  $w_i\in {\mathbb R}^{n_w}$ a standard  Brownian motion, and   where $\{w_i, 1\leq i\leq N\}$ are independent processes. For simplicity,  all collections of system initial conditions are taken to be independent and  have finite second moment.
The cost of agent ${\cal A}_i$ is given by
\begin{equation}
\label{CNP:Minor:GenCost}
J_{i}^{N} (u_i ,u_{-i}) = E \int_0^T \frac{1}{N}\sum_{j=1}^N   l(x_i (t),u_i(t), x_j (t))dt,
\end{equation}
 where   $l(\cdot)$ is the  pairwise running cost, and $u_{-i}$ denotes the controls of all other agents.

The dynamics of a generic agent ${\cal A}_i$ in the infinite population limit of this system is then described by the controlled  McKean-Vlasov (MV) equation
\begin{equation}
\label{MVNLMinorDyn1}
  dx_i = f[x_i,u_i, \mu_t]  dt +\sigma dw_i, \quad
0 \leq t \leq T,
\end{equation}
where  $\mu_t$ is the distribution of  $x_i(t)$, $f[x,u,\mu_t] \coloneqq  \int_{\mathbb{R}^n} f(x,u,y)\mu_t(dy)$ and where the initial distribution $\mu^x_0$ of $x_i(0)$ is specified.
Setting  $l[x,u,\mu_t]=
  \int_{\mathbb{R}^n}l(x,u,y)\mu_t(dy)$,  the corresponding infinite population  cost  for   $ {\cal {A}}_{i}$  takes the form
 \begin{equation}
\label{INFPOPPERF}
  J_{i}(u_i;\mu(\cdot))
  \coloneqq  E\int_0^T l[x_i (t), u _i(t),\mu_t] dt.
 \end{equation}

For notational simplicity, we present the graphon MFG framework with scalar individual states and controls, i.e., $n=n_u=n_w=1$. Its extension to the vector case is evident.

Now we consider a finite population distributed over the finite graph $G_k$.
Let $\mathbold{x}_{ G_k}=\bigoplus_{l=1}^{M_k} \{x_i|i\in \maC_l\}$ denote the states of all agents in the total set of clusters of the population. This gives a total  of $N=\sum_{l=1}^{M_k} |\maC_l|$ individual states.
The key feature of  the graphon MFG construction beyond the standard MFG scheme is that at any agent in a  network  the  averaged dynamics \eqref{NLMinorDynamics} and cost function
 \eqref{CNP:Minor:GenCost} decompose into averages of  subpopulations  distributed at that agent's neighboring  nodes plus  an average term for the local cluster.  In the limit, the summed subpopulation averages are given by an integral over the local mean fields of the neighbouring agents.

For  ${\cal A}_i$ in the cluster $\maC(i)$, two  coupling terms
in the dynamics take the form
\begin{align}
& f_0(x_i, u_i, {\mathcal C}(i)) =\frac{1}{|{\mathcal C}(i)|}\sum_{j\in {\mathcal C}(i)} f_0(x_i, u_i, x_j),  \\
& f_{G_k} (x_i,u_i, g^k_{\maC(i)})=\frac{1}{M_k}
\sum_{l=1}^{M_k} g^k_{\maC(i)\maC_l} \frac{1}{|\maC_l|}
\sum_{j\in \maC_l}f(x_i, u_i, x_j). \label{intercluster}
\end{align}
They model intra- and inter-cluster couplings, respectively.
The specification of $f_{G_k}$ relies on the sectional information
$g^k_{\maC(i)\bullet}$.
Concerning the coupling structure in \eqref{intercluster} we observe that with respect to ${\cal A}_i$, all individuals residing in cluster $\maC_l$ are symmetric and their
state average generates the overall impact of that cluster
on ${\mathcal A}_i$ mediated by the graphon weighting $g^k_{\maC(i)\bullet}$.   The two coupling terms are combined additively resulting in the local dynamics
$$
\tilde f_{G_k}(x_i, u_i, g^k_{\maC(i)})=
 f_0(x_i, u_i,{\mathcal C}(i))+  f_{G_k}(x_i, u_i,g^k_{\maC(i)} ).
$$
Note that ${\mathcal A}_i$  interacts with the overall population through a function of the complete system state
${\mathbold x}_{G_k}$ and the cluster sizes. These details shall be suppressed in this paper and we  only indicate the graph $G_k$ and the section $g^k_{\maC(i)}$.
The state process of  ${\cal A}_i$ is then given by the stochastic differential equation
\begin{align}
dx_i(t) = \tilde f_{G_k}(x_i, u_i,g^k_{\maC(i)} ) dt +\sigma
dw_i, \quad 1\le i\le N, \nonumber
\end{align}
where $\sigma>0$ and the initial states $\{x_i(0), 1\le i\le N\}$ are i.i.d. with distribution $\mu^x_{0}\in {\mathcal P}_1(\mathbb{R})$, the set of  probability measures on $\mathbb{R}$ with finite mean.

The  limit of the two dynamic coupling terms of an agent at a node $\alpha$
(called an $\alpha$-agent),   as the number of nodes of the graph $G_k$ and the subpopulation at each node tend to infinity, is described by the following expressions:
\begin{align}
& {f}_0[x_{\alpha}, u_{\alpha}, {\mu}_{\alpha}] \coloneqq  \int_{{\mathbb R}}  f_0(x_{\alpha},u_{\alpha}, z)  \mu_{\alpha} (dz),
\label{f0brac}\\
& {f}[x_{\alpha}, u_{\alpha}, {\mu}_{G}; g_{\alpha}] \coloneqq  \int_0^1\int_{{\mathbb R}}  f(x_{\alpha},u_{\alpha}, z) g(\alpha, \beta) \mu_{\beta} (dz)d\beta,
\label{fbrac}
\end{align}
 which give the  complete local graphon dynamics via
\begin{equation}
\label{NMFGdynamics}
         \widetilde{f}[x_{\alpha}, u_{\alpha}, {\mu}_{G}; g_{\alpha}]
\coloneqq f_0[x_{\alpha}, u_{\alpha}, \mu_\alpha ] +  {f}[x_{\alpha}, u_{\alpha}, {\mu}_{G}; g_{\alpha}].
\end{equation}
We call $\mu_\beta$ the local mean field at node $\beta$, which is interpreted as the limit of the empirical distributions of agents at node $\beta$.
And $\mu_G=\{\mu_\beta, 0\le \beta\le 1\}$ is the ensemble of local mean fields.
Due to the integration with respect to $\beta$, the dependence of $\widetilde f$ on the graphon limit $g$ is through the section $g_\alpha$.
Since $\mu_G$  contains $\mu_\alpha$, we do not list $\mu_\alpha$ as an argument of $\widetilde f$.

Parallel to the standard MFG case, in the graphon case  the stochastic differential equation
\begin{equation}
\label{MVNLMinorDyn2}
\begin{aligned}
&{\text{[MV-SDE]}}
 (\alpha) \quad  dx_\alpha(t) =  \widetilde{f}[x_{\alpha}(t), u_{\alpha}(t), {\mu}_{G}(t); g_{\alpha}] dt +\sigma dw_\alpha(t),  \\
& \quad 0 \leq t \leq T,\quad
 \alpha \in [0,1],
 \end{aligned}
\end{equation}
generalizes the standard controlled MV equation \eqref{MVNLMinorDyn1}.
We note that in a parallel development of graphon based stochastic dynamical populations \cite{BCW20} the system disturbance intensity $\sigma$ is also a function of  graphon weighted state functions at other clusters.
For simplicity, we consider a constant $\sigma$ and our analysis may be generalized to the case of a state and mean field dependent diffusion term. Similarly, for simplicity our dynamics and cost do not include a separate parametrization by $\alpha$.

Analogously, in the GMFG case, we define the cost
 coupling terms  for ${\cal A}_i$  to be
\begin{align*}
& l_0(x_i, u_i,\maC(i) )= \frac{1}{| \maC(i)|}\sum_{j\in \maC(i)} l_0(x_i, u_i, x_j),  \\
& l_{G_k}(x_i, u_i,g^k_{\maC(i)} )= \frac{1}{M_k}
\sum_{l=1}^{M_k} g^k_{\maC(i)\maC_l} \frac{1}{|\maC_l|}
\sum_{j\in \maC_l}l(x_i, u_i, x_j).
\end{align*}
Define
$
\tilde l_{G_k}(x_i, u_i, g^k_{\maC(i)}) = l_0(x_i, u_i, \maC(i)) + l_{G_k}(x_i, u_i, g^k_{\maC(i)} ) .
$
The cost of ${\mathcal A}_i$ in a finite population on a finite
graph $G_{k}$ is  given in the form
\begin{align}
 J_i= E\int_0^T \tilde l_{G_k}(x_i, u_i,g^k_{\maC(i)}  ) dt.  \label{coGk}
 \end{align}
Denote
 \begin{align}
&   l_0[x_\alpha, u_\alpha,\mu_\alpha]=\int_{\mathbb{R}} l_0 (x_\alpha, u_\alpha, z)  \mu_\alpha(dz) ,
\nonumber   \\
&  l[x_\alpha, u_\alpha,\mu_G; g_\alpha]=\int_0^1 \int_{\mathbb{R}}
l (x_\alpha, u_\alpha, z) g(\alpha, \beta) \mu_\beta(dz) d\beta,
\nonumber\\
&\widetilde l [x_\alpha, u_\alpha, \mu_G; g_\alpha ]
=l_0[x_\alpha,u_\alpha, \mu_\alpha]+l[x_\alpha, u_\alpha,\mu_G; g_\alpha].\nonumber
\end{align}
Then in  the infinite population
graphon case,
the  $\alpha$-agent has the cost function given by
\begin{align}
J_\alpha(u_\alpha; \mu_G(\cdot))=
E\int_0^T \widetilde l [x_\alpha(t), u_\alpha
 (t), \mu_G(t); g_\alpha] dt .  \label{Jalpha}
\end{align}

\subsection{The Graphon MFG Model and Its Equations}

In this section the  standard MFG equations (see e.g. \cite{caines2015mean,CHM17})  will be  generalized so that they subsume the standard (implicitly uniform totally connected) dense network case and cover the fully general graphon limit network case. Specifically,   agent  ${\cal A}_i$ in a population of $N$ agents
will be located at the $l$th node in an $M_k$ node network (identified with its graphon) and in the infinite population graphon limit that node will be taken to map to  $\alpha \in [0,1]$. It is important to note here that {\it although the limit network is assumed dense it is not assumed to be uniformly totally connected}; indeed, the connection structure of the infinite network is represented precisely by its graphon  $ g(\alpha, \beta)$, $0 \le \alpha, \beta \le 1.$

The generalized Graphon MFG  scheme below on $[0,T]$ is given  for
each $\alpha$ by  (i) the Hamilton-Jacobi-Bellman  (HJB) equation generating the value function $V^\alpha$ when all other agents' control laws and the  ensemble $\mu_G$  of local mean fields are given,  (ii) the FPK equation generating the local mean field $\mu_\alpha$ given $\mu_G$, and  (iii) the specification of the best response (BR) feedback law.

 Suppressing the time index on the measures for simplicity of notation, we have the {\it Graphon Mean Field Game (GMFG) equations}:
\begin{align}
\label{MFGPDES}
& {\text{[HJB]}}(\alpha)
 \quad -\frac{\p V^{\alpha  }(t,x)}{\p t}  = \inf_{u\in U}\bigg\{\widetilde{f}[x ,u, {\mu}_{G};g_{\alpha} ]
\frac{\p V^{\alpha}(t,x)}{\p x}\nonumber\\
&\hskip 4cm+ \widetilde{l}[x,u, {\mu}_{G }; g_{\alpha} ]\bigg\}
 +\frac{\sigma^2}{2} \frac{\p^2 V^{\alpha  }(t,x)}{\p x^2},\\
 &V^{\alpha  }  (T,x) =0,    \quad(t,x)\in [0,T]\times\mathbb{R}, \quad  \alpha \in [0,1], \nonumber
  \end{align}
  \begin{align}
 \label{FPK}
{\text{[FPK]}}(\alpha)  \quad
\frac{\p p_\alpha(t,x)}{\p t}  =&-\frac{\p \{ \widetilde{f}[x,    u^0, \mu_{G };g_{\alpha} ] p_\alpha(t,x) \} }{\p x} \nonumber \\
& + \frac{\sigma^2}{2}\frac{\p^2 p_\alpha(t,x)}{\p x^2},
\end{align}
\begin{align}
{ \text{[BR]}}(\alpha) \quad   u^0
& \coloneqq \varphi(t,x | {\mu}_{G};g_{\alpha}  ). \nonumber
\end{align}
Here  $p_\alpha(t,x)$ with initial condition $p_\alpha(0)$ is used to denote the density of the measure $ \mu _\alpha(t)$ whenever a density is assumed to exist. The FPK equation may be replaced by the following closed-loop MV-SDE:
\begin{align}
\text{[MV]}(\alpha)\quad dx_\alpha(t) =  \widetilde{f}[x_{\alpha}(t),\varphi(t,x_\alpha(t) | {\mu}_{G};g_{\alpha}  ) , {\mu}_{G}(t); g_{\alpha}] dt +\sigma dw_\alpha(t),
\label{clMV} \end{align}
where $x_\alpha(0)$ has distribution $\mu_0^x$.
 Our subsequent analysis will directly treat the pair $(V^\alpha(t,x), \mu_\alpha(t))$, where $\mu_\alpha(t)$ is specified as the law of $x_\alpha(t)$ in \eqref{clMV}.

When a solution exists for the GMFG equations, the resulting BR feedback controls depend upon the ensemble $\mu_G$ of local mean fields  and the individual agent's  state. This is a natural generalization of the standard case.
The  standard  MFG case is simply obtained by setting $   g(\alpha, \beta) \equiv 0,  0 \le \alpha, \beta \le 1$, which  totally disconnects the network and results  in $ \widetilde{f}[x,u, {\mu}_G ; g_\alpha] =  {f}_0[x,u, \mu]$ and $ \widetilde{l}[x,u, {\mu}_G ; g_\alpha] =  l_0[x,u,\mu]$
\citep{caines2015mean, CHM17}.

A collection of
measures on some measurable space which are indexed by  the vertex set $[0,1]$
is called a measure ensemble.
Thus, for each fixed $t$, $\mu_G(t)$ is a  measure ensemble.

On ${\mathcal P}_1({\mathbb R})$  we endow the Wasserstein metric $W_1$: for any $\mu, \nu\in {\mathcal P}_1({\mathbb R})$,  $W_1(\mu, \nu)=\inf_{\widehat \gamma} \int|x-y|\widehat \gamma (dx,dy) $, where $\widehat \gamma$ is a probability measure on $\mathbb{R}^2$ with marginals $\mu, \nu$.

Let $C([0,1],{\mathcal P}_1({\mathbb R}))$ be the set of   measure ensembles   $\nu_G=(\nu_\beta)_{\beta\in [0,1]}$ satisfying   $\nu_\beta \in {\mathcal P}_1({\mathbb R})$, and $\lim_{\beta'\to \beta }W_1(\nu_{\beta'}, \nu_\beta)=0$  for any $\beta\in [0,1]$.

In order to analyze the solvability of the GMFG equations, we need to restrict $\mu_G(\cdot)$ to a certain class. We say
$\{\mu_G(t), 0\le t\le T \}$ is from the admissible set  ${\cal M}_{[0,T]}$ if:

(C1) For each  fixed $t$,
  $\mu_G(t)$ is in $C([0,1],{\mathcal P}_1({\mathbb R}))$.

(C2) There exists $\eta\in (0, 1]$ such that for any bounded and Lipschitz continuous function $\phi$ on $\mathbb{R}$,
$$
\sup_{\beta\in [0,1]}\Big|\int_\mathbb{R} \phi(y) \mu_\beta (t_1, dy) -
\int_\mathbb{R} \phi(y) \mu_\beta(t_2, dy)\Big|\le C_h |t_1-t_2|^\eta,
$$
where $C_h$ may be selected to depend only on  the Lipschitz constant $\mbox{ Lip}(\phi)$ for $\phi$.

Condition (C1) ensures  that integration with respect to $d\beta$ in \eqref{fbrac} is well defined.
Condition (C2) ensures that the drift term in the HJB equation \eqref{MFGPDES} has a certain time continuity, which  facilitates the subsequent existence analysis of the best response.

\subsection{Existence Analysis}

We introduce the following assumptions:

(H1) $U$ is a compact set.

(H2) $f_0(x,u,y)$,  $f(x,u,y)$, $l_0(x,u,y)$ and $l(x,u,y)$ are continuous and bounded functions on $\mathbb{R}\times U\times \mathbb{R}$ and  are Lipschitz continuous in $(x,y)$,  uniformly with respect to $u$.


(H3) $f_0(x,u,y)$ and $f(x,u,y)$ are Lipschitz continuous in $u$, uniformly with respect to $(x,y)$.

(H4) For any $q\in \mathbb{R}$, $\alpha \in [0,1]$   and  probability measure ensemble  ${\nu}_{G}\in C([0,1],{\mathcal P}_1({\mathbb R}))$, the set
\begin{align}
S_\alpha^{\nu_G}(x, q)&= \arg\min_{u \in U} \{ q( \widetilde{f}[x, u, {\nu}_{G}; g_{\alpha}]) + \widetilde{l}[x,u, {\nu}_{G };g_{\alpha}] \}
\end{align}
is a singleton, and for any given compact interval ${\mathcal I}=[\underline{q}, \bar q]$, the resulting $u$ as a function of $(x,q)\in \mathbb{R}\times {\mathcal I}$ is Lipschitz continuous in $(x,q)$, uniformly with respect to ${\nu}_{G }$ and $g_{\alpha}$, $0\le \alpha\le 1$.

The next two assumptions will be used to ensure that the best responses have continuous dependence on $\alpha$.
In particular, (H5) is a continuity assumption  on the graphon function $g(\alpha, \beta)$.    Under (H5),  $\widetilde f$ and $\widetilde l$ have continuity in $\alpha$.

(H5) For any bounded and measurable function $h(\beta)$, the function
$\int_0^1g(\alpha, \beta) h(\beta)d\beta$
is continuous in $\alpha \in [0,1]$.

(H6)  For given $\nu_G \in C([0,1],{\mathcal P}_1({\mathbb R}))$,  $S_\alpha^{\nu_G}(x, q) $ is continuous in $(\alpha,x,q)$.

Although the GMFG equation system only involves
$\{\mu_G(t), 0\le t\le T\}$,
which may be viewed as a collection of marginals at different vertices, it is necessary to develop the existence analysis in the underlying probability spaces (see related discussions in
\cite[p.240]{HMC06}).

We begin by introducing some analytic preliminaries.  For the space $C_T=C([0,T], \mathbb{R})$, we specify a $\sigma$-algebra ${\cal F}_T$ induced by all cylindrical sets of the form
$\{x(\cdot)\in C_T: x(t_i)\in B_i, 1\le i\le j \mbox{ for some } j
\}$, where $B_i$  is a Borel set. Let ${\bf M}_T$ denote the space of all probability measures on $(C_T, {\cal F}_T)$.
The canonical process $X$ is defined by $X_t(\omega)=\omega_t$ for $\omega\in C_T$.
On $C_T$, we introduce the metric $\rho(x, y)=\sup_t|x(t)-y(t)|\wedge 1$. Then $(C_T, \rho)$ is a complete metric space. Based on $\rho$, we introduce the Wasserstein metric on ${\bf M}_T$. For $m_1, m_2\in {\bf M}_T$, denote
\begin{align}
&D_T(m_1, m_2) =
\inf_{\widehat m} \int_{C_T\times C_T} \Big(\sup_{s\le T} |X_s(\omega_1)-X_s(\omega_2)|\wedge 1\Big) d\widehat m(\omega_1, \omega_2), \nonumber
\end{align}
where $\widehat  m$ is called a coupling as a probability measure on $(C_T,{\cal F}_T)\times (C_T, {\cal F}_T)$ with the pair of marginals  $m_1$ and $m_2$, respectively. Then $({\bf M}_T, D_T)$ is a complete metric
space \cite{szn91}.

 We introduce the  product of probability measure spaces
$\prod_{\alpha \in [0,1]} (C_T, {\cal F}_T, m_\alpha)$,
where each individual space is interpreted as the
path space of the agent at vertex $\alpha$ with a corresponding probability measure $m_\alpha$. Denote the product of spaces of probability measures
$
{\bf M}_{T}^G=\prod_{\alpha\in [0,1]} {\bf M}_T.
$
An element in ${\bf M}_{T}^G$ is  a measure ensemble.
Given  $m_G\in {\bf M}_{T}^G$,  the projection operator ${\rm Proj}_\alpha$ picks out its component $m_\alpha$ associated with $\alpha\in [0,1]$.  Let ${\bf M}_T^{G0}$ consist of all $(m_\alpha)_{\alpha\in [0,1]}\in{\bf M}_T^{G} $ such that for any $\alpha\in [0,1]$,
 $D_T(m_{\alpha'}, m_{\alpha})\to 0$ as $ \alpha'\to \alpha$.

For two measure ensembles $m_G\coloneqq(m_\alpha)_{\alpha\in [0,1]}$ and $\bar m_G\coloneqq (\bar m_\alpha)_{\alpha \in [0,1]}$ in ${\bf M}_{T}^G$, define
$
d(m_G, \bar m_G)= \sup_{\alpha\in [0,1]} D_T(m_\alpha, \bar m_\alpha).
$

  \begin{lemma}
 $({\bf M}_{T}^{G}, d)$ is a  complete metric space.
\end{lemma}

\proof If $\{m_G^k, k\ge 1\}$ is a Cauchy sequence in ${\bf M}_{T}^G$,
then for each given $\alpha$, the sequence $\{{\rm Proj}_\alpha (  m_G^k), k\ge 1 \}
$ (of probability measures)
is a Cauchy sequence in the complete metric space ${\bf M}_T$
and so it contains  a limit. This in turn  determines a limit in
${\bf M}_{T}^G$. \endproof

 Given the probability measure $m_\alpha\in {\bf M}_T$, we determine the $t$-marginal $\mu_\alpha(t) $ by
$\mu_\alpha(t, B)=m_\alpha (\{x(\cdot)\in C_T: x(t)\in B\})$ for any Borel set $B\subset \mathbb{R}$, and denote the mapping from ${\bf M}_T$ to ${\mathcal P}(\mathbb{R})$ (the set of probability measures on ${\mathbb R}$):
\begin{align}
 \mu_\alpha(t)={\rm Marg}_t (m_\alpha ). \label{marjma}
\end{align}
Consider the measure ensemble  $m_G=(m_\alpha)_{\alpha \in [0,1]}\in {\bf M}_{T}^G$ with $\mu_\alpha(t)$ given by \eqref{marjma}. Define the time $t$ marginals  by the following mapping
\begin{align}
{\rm Marg}_t (m_G)= (\mu_\alpha(t))_{\alpha \in [0,1]}   , \label{mualt}
\end{align}
where the right hand side is simply written as $\mu_G(t) $.
For a given $t$,  $\mu_G(t) $ may be interpreted as a measure valued function defined on  the vertex set $[0,1]$.
Further denote the mapping
$ {\rm Marg} (m_G) = (\mu_G(t))_{t\in [0,T]} = \mu_G(\cdot)$.

Take a fixed
\begin{align}
 \mu_G(\cdot) \in {\cal M}_{[0,T]}\label{muGMfix}
 \end{align}
 with its associated H\"older parameter $\eta$ in (C2), and denote
 \begin{align*}
\widetilde f_\alpha^*(t, x, u)=\widetilde f[x,u,\mu_G(t) ; g_\alpha],
\quad \widetilde l_\alpha^*(t, x, u)=\widetilde l[x,u,\mu_G(t) ; g_\alpha]. \end{align*}

\begin{lemma}\label{lemma:HL} Assume \emph{(H1)--(H2)}.
For $h_\alpha=\widetilde f_\alpha^*(t, x, u)$ or $  \widetilde l_\alpha^*(t, x, u)  $,  there exist   constants $C$ and $C_{\mu_G}$, where the latter  depends on $\mu_G(\cdot)$,  such that
\begin{align*}
&\sup_{t,u,\alpha}|h_\alpha(t, x, u) -h_\alpha(t, y, u)|\le C|x-y|,\\
&\sup_{x,u,\alpha}|h_\alpha(t, x, u) -h_\alpha(s, x, u)|\le C_{\mu_G}|t-s|^\eta, \end{align*}
where the supremum is taken over $t\in [0,T]$, $x\in \mathbb{R}$, $u\in U$ and $\alpha \in [0,1]$.
\end{lemma}
\begin{proof} The Lipschitz continuity of $\widetilde f^*_\alpha$ with respect to $x$ follows from (H2) and \eqref{f0brac}--\eqref{fbrac}.
For $t_1, t_2\in [0,T]$, we
estimate $|\widetilde f[x,u, \mu_G(t_1); g_\alpha] -\widetilde  f[x, u,\mu_G(t_2); g_\alpha] |$ by using the Lipschitz condition of $f_0$, $f$ and condition (C2) for ${\cal M}_{[0,T]}$. This establishes the H\"older continuity of
$\widetilde f_\alpha^*$ in $t$. The other cases can be similarly checked.
\end{proof}

In order to analyze the best response of the $\alpha$-agent, we
introduce the HJB equation
\begin{align}
-V^\alpha_t(t,x)= \inf_{u\in U} \{ \widetilde f_\alpha^*(t, x, u) V_x^\alpha (t,x)+\widetilde l_\alpha^*(t,x, u) \} +\frac{\sigma^2}{2} V_{xx}^\alpha (t,x), \label{hjbmug}
\end{align}
where $V^\alpha (T,0)=0$. It differs from \eqref{MFGPDES} by allowing an arbitrary $\mu_G(\cdot)\in {\mathcal M}_{[0,T]}$.

For studying \eqref{hjbmug}, we introduce some standard definitions.
Denote $Q_T= (0,T)\times \mathbb{R}$,
and $\overline Q_T = [0,T]\times \mathbb{R}$.
Let $C^{1,2} (\overline Q_T)$ (resp., $C^{1,2} ( Q_T)$) denote
the set of functions with continuous derivatives $v_t, v_x, v_{xx}$ on
$ \overline Q_T $ (resp., $Q_T$).
Let $C_b^{1,2} (\overline Q_T)$ be the set  of bounded functions in
$C^{1,2} (\overline Q_T)$, and
let the open (or closed) set $Q_b$ be a bounded subset of $Q_T$.
 $W^{1,2}_{\lambda}(Q_b)$, $1\le \lambda < \infty $, shall denote the Sobolev space consisting of functions $v$ such that each $v$ and its  generalized derivatives $v_t$, $v_x$, $v_{xx}$ are in $L^\lambda (Q_b)$; further we have the norm
\begin{align}
\|v\|_{\lambda, Q_b}^{(2)} = \|v\|_{\lambda, Q_b} + \|v_t\|_{\lambda, Q_b} +\|v_x\|_{\lambda, Q_b} +\|v_{xx}\|_{\lambda, Q_b}, \label{sobn}
\end{align}
where $\|v\|_{\lambda, Q_b} = (\int_{Q_b} |v(t,x)|^\lambda dt dx)^{1/\lambda}$.
Set $|v|_{Q_b}=\sup_{(t,x)\in Q_b} |v(t,x)|$.
Now for $Q_b=(T_1,T_2)\times {\cal I} $, where ${\cal I}$ is a bounded open subset of $\mathbb{R}$, and  $\beta\in (0, 1) $, define the  H\"older norms
\begin{align*}
&|v|^\beta_{Q_b}= |v|_{Q_b} + \sup_{ t\in (T_1, T_2),x,y\in {\cal I} } |v(t, x)-v(t,y)|
\cdot | x-y|^{-\beta}\\
&\qquad\qquad +
\sup_{s,t\in (T_1,T_2),x\in {\mathcal I}}|v(s,x)-v(t,x)|\cdot|s-t|^{-\beta/2}, \\
&|v|^{1+\beta}_{Q_b} = |v|^{\beta}_{Q_b} +|v_x|^{\beta}_{Q_b}, \\
&|v|^{2+\beta}_{Q_b} = |v|^{1+\beta}_{Q_b} +|v_t|^{\beta}_{Q_b} + |v_{xx}|^{\beta}_{Q_b}.
\end{align*}

\begin{lemma} \label{lemma:HJBc12}
Under \emph{(H1)--(H4)}, the following holds:

\emph{(i)} Equation \eqref{hjbmug} has a unique solution  $V^\alpha$ in
$C_b^{1,2}(\overline Q_T)$ and moreover $\sup_{\overline Q_T} |V^\alpha_{xx}|\le C$.

\emph{(ii)} 
 The best response
  \begin{align}\label{gmubr}
  u_\alpha = \phi_\alpha(t, x|\mu_G(\cdot)), \quad \alpha\in [0,1]
  \end{align}
as the optimal control law solved from \eqref{hjbmug} is bounded and Borel measurable   on $[0,T]\times \mathbb{R}$, and
 Lipschitz continuous in $x$, uniformly with respect to $\alpha$ for the given $\mu_G(\cdot)$.
\end{lemma}

\begin{proof}
(i) Denote
$$
 \mathbold{H}_\alpha(t, x, q)=\min_{u\in U}\{ q \widetilde f_\alpha^*(t, x, u) +\widetilde l_\alpha^*(t,x, u)\} .
$$
Then \eqref{hjbmug}  may be rewritten as
\begin{align}\label{VHV}
- V^\alpha_t (t,x)  =
\mathbold{H}_\alpha(t, x, V^\alpha_x) + \frac{\sigma^2}{2}
 V^\alpha_{xx}, \qquad  V^\alpha(T,x)=0.
 \end{align}
As in the proof of \cite[Theorem 5]{HMC06}, we use H\"older and Lipschitz continuity (with respect to $t$ and $x$, respectively) of $\widetilde f_\alpha^*$ and $\widetilde l_\alpha^*$ in Lemma \ref{lemma:HL}, and  follow the method in the proof of Theorem VI.6.2 of \cite[p. 210]{FR75} to show that \eqref{hjbmug} has a unique solution $V^\alpha\in C_b^{1,2}(\overline Q_T)$,
where uniqueness follows from a verification theorem using  the closed-loop state process.

Next we show that $V_{xx}^\alpha$ is bounded on $\overline Q_T$.
Take any $x_0\in \mathbb{R}$. Denote $B_{r}(x_0)= (x_0-r, x_0+r)$ for $r>0$, and $Q_T^{x_0,r}=(0,T)\times B_r(x_0)$.
We use two steps involving local estimates. Each step gets refined information about $V^\alpha$ in a region based on available bound information in a
larger region. It suffices to obtain a bound of $V_{xx}^\alpha$ on
$Q_T^{x_0, 1} $ as long as this bound does not change with $x_0$.

Step 1.
First,  there exists a constant $C_1$ such that
\begin{align}
\sup_{t,x,\alpha}|V^\alpha|\le C_1, \quad \sup_{t,x,\alpha}|V_x^\alpha|\le C_1.  \label{VVC1}
\end{align}
The first inequality is obtained using (H1)--(H2) and the fact that $V^\alpha$ is the value function of the associated optimal control problem. The second inequality is proven by the difference estimate of $|V^\alpha(t,x)-V^\alpha(t, y)|$
as in \cite[p. 209]{FR75}.

By (H1), (H2) and \eqref{VVC1}, we have
$$
\sup_{\alpha}\sup_{(t,x)\in \overline Q_T  }|\mathbold{H}_\alpha (t, x, V_x^\alpha(t,x))|\le C_2.
$$

We use a typical method for analyzing  semilinear parabolic equations.
Once $V^\alpha$ is known to be  a solution of \eqref{VHV},
 we view $V^\alpha$ as the solution of a linear equation  with the free term $\mathbold{H}_\alpha (t, x, V_x^\alpha)$.
 For further estimates, we need $\lambda>n+2$ when using the norm \eqref{sobn}.
Fix $\lambda=n+3=4$.
 This yields the  bound
$$
\|V^\alpha\|^{(2)}_{\lambda,Q_T^{x_0, 2}} \le C_3,
$$
where $C_3$ depends on $(C_2, T , \sigma)$ and the bound of $(f, f_0, l, l_0)$ but not on $x_0$, $\alpha$; see \cite[p. 207]{FR75}  and also \cite[p. 342]{Ladyzh68} for local estimates of the Sobolev norm of solutions defined on unbounded domain using a cut-off function.
Take $\beta= 1-\frac{n+2}{\lambda}=\frac{1}{4}$.
Subsequently, since  $\lambda>n+2$, we have the H\"older estimate
\begin{align}
|V^\alpha |_{Q^{x_0, 2}_T}^{1+\beta} \le C_4\| V^\alpha \|^{(2)}_{_{\lambda,Q_T^{x_0, 2}}}\le C_3C_4, \label{VV34}
\end{align}
where $C_4$ is determined by $\lambda =4$ without depending on $x_0, \alpha$; see \cite[p. 207]{FR75}, \cite[p. 343]{Ladyzh68}.

Step 2.  On $[0,T]\times \mathbb{R} \times [-C_1, C_1]$, we can show
${\mathbold H}_\alpha(t, x, q)$ is H\"older continuous in $t$ and Lipschitz continuous in $(x, q)$.  Denote $\beta_1= \min\{\eta, \beta\}$.
Next we view $\mathbold{H}_\alpha (t, x, V_x^\alpha(t,x))$ as  a function of $(t, x)$. Then by use of \eqref{VV34} we further obtain a bound on the H\"older norm:
\begin{align}
\sup_\alpha\sup_{x_0}|\mathbold{H}_\alpha (\cdot, \cdot, V_x^\alpha) |^{\beta_1}_{Q^{x_0, 2}_T} \le C_5.\label{HC5}
\end{align}

Subsequently, by the method in \cite[p. 207-208]{FR75} with its cut-off function technique and \cite[p. 351-352]{Ladyzh68},
we  use \eqref{HC5} and local H\"older estimates of \eqref{VHV} to obtain
\begin{align}
|V^\alpha|^{2+\beta_1}_{Q^{x_0, 1}_T}  \le C_{6 }, \label{Vb6}
\end{align}
where $C_6$ depends on $C_5$ but not on $x_0,\alpha$.
Since $x_0$ is arbitrary,
it follows that
\begin{align}
\sup_\alpha\sup_{\overline Q_T} |V^\alpha_{xx}|\le C_6 . \label{vxxc6}
\end{align}

(ii) By (H4), the optimal control law \eqref{gmubr}
as a  function of $(t, x)$ is well defined  and is bounded on $[0,T]\times {\mathbb R}$ by compactness of $U$. It is Borel measurable on $\overline Q_T$; see \cite[p.168]{FR75}.
Since $S_\alpha^{\nu_G}(x, q)$ is Lipschitz continuous in $(x, q)\in \mathbb{R}\times [-C_1, C_1]$ and $V_x^\alpha(t,x)$ is  Lipschitz continuous in $x\in \mathbb{R}$ by  \eqref{vxxc6}, uniformly with respect to $\alpha$ in each case,  $\phi_\alpha$ is uniformly Lipschitz continuous in $x$.
\end{proof}

Denote
$$
\Psi^\alpha(t,x)= (V^\alpha(t,x),V_t^\alpha(t,x), V^\alpha_x(t,x),
V^\alpha_{xx}(t,x)), \quad (t,x)\in \overline Q_T.
$$
We prove the following continuity lemma for the solution of \eqref{hjbmug}.
For $\overline Q_T$, define the compact subsets $ B_j= \{(t, x)| 0\le t\le T, |x|\le j\}$,  $j\in \mathbb{N}$.

\begin{lemma} \label{lemma:vxcont}
Assume \emph{(H1)--(H5)} hold and let $\mu_G(\cdot)$ in \eqref{muGMfix} be fixed. Then the following holds:

\emph{(i)} For all compact set $ B_{j}$,
$\lim_{\alpha'\to \alpha }|\Psi^{\alpha'}-\Psi^\alpha |_{ B_{j}}=0$.

\emph{(ii)} $\lim_{\alpha'\to \alpha }V_x^{\alpha'}(t,x) =V_x^\alpha (t,x)$ for all $ (t,x)\in [0,T]\times \mathbb{R}$.
\end{lemma}

\begin{proof}
It suffices to show (i) as (ii) follows immediately from (i).

Step 1.
By  \eqref{Vb6} and the fact that the constant $ C_6$  can be selected without depending on $\alpha$,
there exists a constant $C$ such that
$\sup_{\alpha}|V^\alpha|^{2+\beta_1}_{B_j} \le C,
$
 which implies that $\{\Psi^\alpha, \alpha \in [0,1]\}$ is uniformly bounded and equicontinuous on $B_j$.
For any sequence $\{\alpha_k, k\ge 1\}$ converging to $\alpha$, by Ascoli-Arzela's lemma, for $j=1$, there exists a subsequence denoted by $\{\bar \alpha_k, k\ge 1\}$ such that $\Psi^{\bar \alpha_k}$ converges uniformly on $ B_{1}$. By  a diagonal  argument, we may further extract a subsequence of
$\{\bar \alpha_k, k\ge 1\}$, denoted by $\{\hat \alpha_k, k\ge 1\}$, such that  $\Psi^{\hat \alpha_k}$ converges uniformly on each set $ B_{j}$, $j\ge 1$. Hence there exists a function $V^*$ with continuous derivatives $V^*_t, V^*_x, V^*_{xx}$ on $\overline Q_T$ such that
 \begin{align}
 \lim_{k\to \infty}\Psi^{\hat \alpha_k}(t,x)= \Psi^*(t,x), \qquad
 \forall (t,x)\in \overline Q_T,  \label{psias}
 \end{align}
  where $\Psi^*=(V^*, V^*_t, V^*_x, V^*_{xx})$. Since
$$
-V_t^{\hat \alpha_k}(t,x) = {\mathbold H}_{\alpha_k}(t,x, V_x^{\hat \alpha_k})+
\frac{\sigma^2}{2}V_{xx}^{\hat \alpha_k},\quad V^{\alpha_k}(T,x)=0,
$$
it follows from \eqref{psias} that
\begin{align}
-V_t^{*}(t,x) = {\mathbold H}_{\alpha}(t,x, V_x^{*})+
\frac{\sigma^2}{2}V_{xx}^{*}, \qquad V^*(T,x)=0. \nonumber
\end{align}
 We have used the fact that ${\mathbold H}_\alpha(t,x,q)$ is continuous in $\alpha$ due to (H5) and condition (C1) of ${\mathcal M}_{[0,T]}$. It is clear that $V^*=V^\alpha$ by uniqueness of the solution of \eqref{VHV}.
So $\Psi^* =\Psi^\alpha$.
Now it follows that
\begin{align}
\lim_{k\to \infty} |\Psi^{\hat \alpha_k}-\Psi^\alpha |_{B_j}=0, \quad
\forall j.
\end{align}

Step 2. Suppose (i) does not hold so that for some $\hat j$ we have
$|\Psi^{\alpha'}-\Psi^\alpha |_{B_{\hat j}}$ does not converge to 0 as $\alpha'\to \alpha$, which implies that there exist some $\epsilon_0>0$ and a sequence  $\{\alpha_k^0\}$ converging to $\alpha$ such that for each $k$,
\begin{align}
|\Psi^{\alpha^0_k}-\Psi^\alpha |_{ B_{\hat j} } \ge \epsilon_0.\label{pspse0}
\end{align}

Step 3.
 Recall that $\{\alpha_k\}$ in Step 1 is arbitrary as long as it converges to $\alpha$. Now we just take $\{\alpha_k\}$ in Step 1 as $\{\alpha_k^0\}$. By Step 1, there exists a subsequence of $\{\alpha^0_k\}$, denoted by
$ \{\hat \alpha^0_k\} $, such that
$\lim_{k\to \infty} |\Psi^{\hat \alpha_k^0}-\Psi^\alpha |_{B_{\hat j}}=0$,
which  contradicts \eqref{pspse0}. Hence (i) holds.
  \end{proof}

\begin{lemma}\label{lemma:contiBR}
 Assume \emph{(H1)--(H6)}.
For  given $\mu_G(\cdot)\in {\mathcal M}_{[0,T]}$, the best response $\phi_\alpha(t,x|\mu_G(\cdot))$ in \eqref{gmubr}  continuously depends on  $\alpha$. Specifically, for any $\alpha\in [0,1]$,
\begin{align}
\lim_{\alpha'\to \alpha }\phi_{\alpha'}(t,x|\mu_G(\cdot))=\phi_\alpha(t,x|\mu_G(\cdot)),\quad \forall t, x. \label{phiapa}
\end{align}
\end{lemma}

\begin{proof}
The best response can be written as
\begin{align*}
&\phi_\alpha (t, x|\mu_G(\cdot))= S_\alpha^{\mu_G(t)}(x, V_x^\alpha(t,x)),\\
&\phi_{\alpha'} (t, x|\mu_G(\cdot))= S_{\alpha'}^{\mu_G(t)}(x, V_x^{\alpha'}(t,x)).
\end{align*}
It follows that
\begin{align*}
&|S_\alpha^{\mu_G(t)}(x, V_x^\alpha(t,x))-  S_{\alpha'}^{\mu_G(t)}(x, V_x^{\alpha'}(t,x))|\\
\le&  |S_\alpha^{\mu_G(t)}(x, V_x^\alpha(t,x))-  S_{\alpha}^{\mu_G(t)}(x, V_x^{\alpha'}(t,x))|\\
&+ |S_\alpha^{\mu_G(t)}(x, V_x^{\alpha'}(t,x))-  S_{\alpha'}^{\mu_G(t)}(x, V_x^{\alpha'}(t,x))|.
\end{align*}
Given $\mu_G(\cdot)$ we have the prior upper bound  $\sup_{\alpha, t, x}|V_x^\alpha(t,x)|\le C$.
It suffices to show that \eqref{phiapa} holds for any given $C_0>0$ and  $t\in [0,T]$, $|x|\le C_0$. By (H6), for the given $\mu_G(t)$, $S_\alpha^{\mu_G(t)}(x,q)$ is uniformly continuous in $\alpha \in [0,1]$, $|x|\le C_0 $, $q\in [-C, C]$.
For any $\epsilon>0$,  there exists $\delta>0$ such that  $|\alpha-\alpha'|<\delta$ implies
$\sup_{|x|\le C_0,|q|\le C}|S_\alpha^{\mu_G(t)}(x,  q) - S_{\alpha'}^{\mu_G(t)}(x,q)|\le \epsilon/2$, and moreover,
 $$\sup_{|x|\le C_0}|S_\alpha^{\mu_G(t)}(x, V_x^\alpha(t,x))-
  S_{\alpha}^{\mu_G(t)}(x, V_x^{\alpha'}(t,x))|\le \frac{\epsilon}{2}
 $$
  in view of  Lemma \ref{lemma:vxcont} (i).
 Therefore \eqref{phiapa} holds.
\end{proof}

We proceed to show the existence of  a solution to the GMFG equations
\eqref{MFGPDES} and \eqref{clMV}
in terms of $\{(V^\alpha, \mu_\alpha(\cdot))|\alpha\in [0,1]\}$.
 For $\mu_G\in{\cal M}_{[0,T]} $, denote the mapping
$$
 (\phi_\alpha)_{ \alpha\in [0, 1]}  \coloneqq   \Gamma(\mu_G(\cdot)) ,
$$
where the left hand side is given by \eqref{gmubr} as the set of best responses with respect to $\mu_G(\cdot)$. Next, we combine $(\phi_\alpha)_{ \alpha\in [0, 1]}$ with $\mu_G(\cdot)$ to determine the distribution $m_\alpha$ of the closed-loop state process
\begin{align}
dx_\alpha(t)= \widetilde f[x_\alpha(t), \phi_\alpha(t,x_\alpha(t)|\mu_G(\cdot)), \mu_G(t); g_\alpha]dt +\sigma dw_\alpha(t), \nonumber
\end{align}
where $x_\alpha(0)$ has distribution $\mu_0^x$. The choice of the Brownian motion for $x_\alpha$ is immaterial.
 For $m_\alpha$ above, denote the mapping from $ {\cal M}_{[0,T]}$ to $ {\bf M}_T^G $:
$$
(m_\alpha)_{\alpha \in [0,1]} = \widehat
\Gamma  (\mu_G(\cdot)) .
$$

Define the set
\begin{align}
{\bf M}_T^{G1}\coloneqq \widehat \Gamma  ({\cal M}_{[0,T]})\subset
{\bf M}_T^G. \nonumber
\end{align}
Now the existence analysis  may be
formulated as the problem of finding a fixed point of the form
\begin{align}
m_G =\widehat \Gamma \circ {\rm Marg} (m_G),\label{mfp0}
\end{align}
in case $m_G\in {\bf M}_T^{G1}$. Note that
${\rm Marg} (m_G)  
 =\{({\rm Marg}_t
 (m_\alpha))_{\alpha\in [0,1]}, 0\le t\le T \}$.

\begin{remark}
The fixed point problem requires $m_G$ to be from the subset
${\bf M}_T^{G1} $ of ${\bf M}_T^G$. If one simply  looks for $m_G\in {\bf M}_T^G$, the resulting $\mu_G(\cdot)=
{{\rm Marg}}(m_G)$ lacks required properties such as H\"older continuity in (C2), and this  will cause difficulties in establishing Lemma  \ref{lemma:HJBc12} for  the HJB equation.
\end{remark}

\begin{lemma}\label{lemma:unilip}
Under \emph{(H1)--(H6)}, the following assertions hold:

\emph{(i)}  ${\bf M}_T^{G1}\subset  {\bf M}_T^{G0}$.

\emph{(ii)} For any  $m_G\in  {\bf M}_T^{G1} $,
$\mu_G(\cdot)\coloneqq {{\rm Marg}}(m_G)\in {\cal M}_{[0,T]}$.

\emph{(iii)} The best response  $\phi_\alpha(t, x|\mu_G(\cdot))$ with $\mu_G(\cdot)$
given in \emph{(ii)} is Lipschitz continuous in $x$, uniformly with respect to $\alpha\in [0,1]$ and $m_G\in {\bf M}_T^{G1}$.
\end{lemma}

\begin{proof}
(i) and (ii) For $m_G\in {\bf M}_T^{G1} $,
there exists $\mu'_G\in {\mathcal M}_{[0,T]}$ such that
$
m_G=\widehat \Gamma  (\mu'_G(\cdot)) .
$
To estimate $D_T(m_\alpha, m_{\bar\alpha})$ and $W_1(\mu_\alpha(t), \mu_{\bar \alpha}(t))$, let $x_\alpha$ and $x_{\bar\alpha}$ be state processes
generated by \eqref{MVNLMinorDyn2}  with $\mu_G'$, the same initial
state and Brownian motion under the control laws $\phi_\alpha(t,x|\mu_G'(\cdot))$
and $\phi_{\bar\alpha}(t,x|\mu_G'(\cdot))$, respectively.
Then $D_T(m_\alpha, m_{\bar\alpha})\le
E\sup_{t\le T}|x_\alpha(t)-x_{\bar\alpha} (t)|$ and $W_1(\mu_\alpha(t), \mu_{\bar \alpha}(t))\le E|x_\alpha(t)-x_{\bar\alpha} (t)   | $.
  Fixing $\bar \alpha$,
  we have
\begin{align}
|x_\alpha(t)-x_{\bar\alpha}(t)|\le & \int_0^t |\widetilde f[x_\alpha(s),
\phi_\alpha(s,x_\alpha(s)|\mu'_G(\cdot)), \mu'_G(s); g_\alpha] \label{xxaab}\\
&\quad - \widetilde f[x_{\bar\alpha}(s), \phi_{\bar\alpha}(s,x_{\bar\alpha}(s)|
\mu'_G(\cdot)), \mu'_G(s); g_{\bar\alpha}]| ds . \nonumber
\end{align}
Denote
\begin{align*}
&\delta_1=| f_0[x_{\bar\alpha}(s),\phi_{\bar\alpha}(s,x_{\bar\alpha}(s)|\mu'_G(\cdot)) ,\mu'_\alpha(s) ]- f_0[x_{\bar\alpha}(s),\phi_{\bar\alpha}(s,x_{\bar\alpha}(s)|\mu'_G(\cdot)) ,\mu'_{\bar\alpha}(s) ]|, \\
&\delta_2= |f[x_{\bar\alpha}(s), \phi_{\bar\alpha}(s,x_{\bar\alpha}(s)|
\mu'_G(\cdot)), \mu'_G(s); g_{\alpha}]- f[x_{\bar\alpha}(s), \phi_{\bar\alpha}(s,x_{\bar\alpha}(s)|
\mu'_G(\cdot)), \mu'_G(s); g_{\bar\alpha}]|.
\end{align*}
Then by \eqref{xxaab} and the Lipschitz continuity in $x$ of $\phi_\alpha$ in Lemma \ref{lemma:HJBc12} (ii), we obtain
\begin{align}
&|x_\alpha(t)-x_{\bar\alpha}(t)|\le C_1\int_0^t |x_\alpha(s)-x_{\bar\alpha}(s)|ds \label{xaab} \\
&+C_2 \int_0^t\{| \phi_{\alpha}(s,x_{\bar\alpha}(s)|
\mu'_G(\cdot)) - \phi_{\bar\alpha}(s,x_{\bar\alpha}(s)|
\mu'_G(\cdot))| +\delta_1(s) +\delta_2(s)\}ds,\nonumber
\end{align}
where $C_2$ depends only on the Lipschitz  constants of $f_0, f$; and $C_1$ does not change with $\alpha$ for the fixed $\mu_G'$.
Since $W_1(\mu'_\alpha(s), \mu'_{\bar\alpha}(s)  )\to 0$ as
$\alpha\to \bar \alpha$, by (H2) $E\delta_1(s)\to 0$ as $\alpha\to \bar\alpha$. By (H5),  we have $ E\delta_2(s)\to 0$ as $\alpha\to \bar \alpha$.
Then using Lemma \ref{lemma:contiBR} and boundedness of the integrand below, we obtain $$
\lim_{\alpha\to \bar \alpha}E\int_0^T\{| \phi_{\alpha}(s,x_{\bar\alpha}(s)|
\mu'_G(\cdot)) - \phi_{\bar\alpha}(s,x_{\bar\alpha}(s)|
\mu'_G(\cdot))| +\delta_1(s) +\delta_2(s)\}ds=0.
$$
By Gronwall's lemma and \eqref{xaab}, it follows that
\begin{align}
\lim_{\alpha\to \bar \alpha}E\sup_{0\le t\le T}|x_\alpha(t)-x_{\bar\alpha}(t)|=0.
\end{align}
Subsequently, as $\alpha\to \bar \alpha$, we obtain
$D_T(m_\alpha, m_{\bar\alpha})\to 0$, which implies (i); in addition,
$W_1(\mu_\alpha(t), \mu_{\bar \alpha}(t))\to 0$,
which  verifies condition (C1) of ${\cal M}_{[0,T]}$ for $\mu_G$. Since each $m_\alpha$ is the distribution
of $x_\alpha$, for $\mu_G(\cdot)$ we  take the H\"older parameter  $\eta=1/2$
and a constant $C_h$
independent of $\mu_G'$ for (C2). So (ii) holds.

(iii) Due to the choice of $\eta$ and $C_h$ for $\mu_G(\cdot)$ in (ii), we may select a fixed constant $C_5$ in \eqref{HC5}, which does not change with $(\alpha, \mu_G(\cdot))$. Subsequently the upper bound $C_6$ in \eqref{vxxc6}  for $|V_{xx}^\alpha|$ does not change with $\alpha\in [0,1], \mu_G(\cdot)\in {\rm Marg}( \widehat\Gamma ({\mathcal M}_{[0,T]}))$.
This ensures a uniform bound for the Lipschitz constant for $x$ in  $\phi_\alpha$.
\end{proof}

We introduce the  sensitivity condition.

(H7)  For $m_G, \bar m_G\in {\bf M}_T^{G1}=  \widehat \Gamma  ({\cal M}_{[0,T]})$, there exists a constant $c_1$ such that
\begin{align}
\label{RA}
\sup_{t,x, \alpha}|\phi_\alpha(t, x|\mu_G(\cdot))- \bar \phi_\alpha(t,x|\bar \mu_G(\cdot))|\le c_1 d(m_G, \bar m_G ),
\end{align}
where the set of control laws $\{\phi_\alpha(t,x|\mu_G(\cdot)), \alpha\in [0,1]\}$ (resp., $\{\bar\phi_\alpha(t,x|\bar\mu_G(\cdot)), \alpha\in [0,1]\}$) is determined by use of
$\mu_G={\rm Marg}(m_G)$ (resp., $\bar\mu_G={\rm Marg}(\bar m_G)$) in the optimal control problem specified by  \eqref{MVNLMinorDyn2}   and \eqref{Jalpha}   with the graphon section $g_\alpha$.

Assumption (H7) is a generalization from the finite type model in
\cite{HMC06} where an illustration via a linear model is presented.
Related sensitivity conditions are  studied in \cite{KY15}.

Let $(\phi_\alpha)_{ \alpha\in [0,1]}$ in \eqref{gmubr} be applied by all agents, where $\mu_G(\cdot) \in  {\mathcal M}_{[0,T]}$.
We consider the following generalized McKean-Vlasov  equation
\begin{align}
dx_\alpha(t) =\widetilde f [x_\alpha(t), \phi_\alpha (t, x_\alpha(t)|\mu_G),\nu_G(t); g_\alpha ] dt +\sigma dw_\alpha(t), \label{xmunu}
\end{align}
where $x_\alpha(0)$ is given with distribution $\mu_0^x$.
For this equation, $\nu_G$ is  part of the solution.
If $\nu_G$ is determined, we have a unique solution $x_\alpha$ on $[0,T]$ which further determines its law as the measure $m_\alpha$ on  $(C_T, {\cal F}_T)$. Note that $m_\alpha$ does not depend on the choice of the standard Brownian motion $w_\alpha$.  We look for $\nu_G\in {\cal M}_{[0,T]}$ to satisfy the  condition:
\begin{align}
{\rm Marg}_t(m_\alpha) = \nu_\alpha(t),\quad \forall\alpha \in [0,1],\ t\in [0,T], \label{Mnu}
\end{align}
i.e., $\nu_\alpha(t)$ is the law of $x_\alpha(t)$ for all $\alpha,t$
(and we say $(x_\alpha)_{0\le \alpha\le 1}$ is consistent with $\nu_G$).

\begin{lemma}
\label{lemma:phinu} Assume \emph{(H1)--(H6)}.
For the best response control law $\phi_\alpha (t, x_\alpha| \mu_G(\cdot))$ in \eqref{gmubr},
where $\mu_G(\cdot) \in  {\mathcal M}_{[0,T]}$,
there exists a  unique
$ \nu_G(\cdot)$ for \eqref{xmunu} satisfying \eqref{Mnu}.
\end{lemma}

\begin{proof}

In order to solve $(x_\alpha, \nu_G)$ in  \eqref{xmunu}, we specify the law of
the process $x_\alpha$ instead of just its marginal $\nu_\alpha(t)$.
This extends the fixed point idea for treating standard McKean-Vlasov equations \cite{szn91}.

For $(m_\alpha)_{\alpha\in [0,1]}\in{\bf M}_T^{G0} $  , we determine
$\nu_G^1$ according to $\nu^1_\alpha(t)={\rm Marg}_t(m_\alpha)$, which is used in \eqref{xmunu} by taking $\nu_G=\nu_G^1$ to solve $x_\alpha$ on $[0,T]$. Let $m_\alpha^{\rm new}$ denote the law of $x_\alpha$. It in general does not satisfy ${\rm Marg}_t(m_\alpha^{\rm new}) = \nu_\alpha (t)$ for all $t$.
Denote the mapping
$$
(m_\alpha^{\rm new})_{ \alpha\in [0, 1]} =\Phi_{{\bf M}_T^{G0}}( (m_\alpha)_{\alpha\in [0,1]} ).
$$
By (H5) and Lemma \ref{lemma:contiBR},  $\Phi_{{\bf M}_T^{G0}}$ is a mapping from
 ${\bf M}_T^{G0}$ to itself.
Similarly, from $ (\bar m_\alpha)_{ \alpha\in [0,1 ]} \in {{\bf M}_T^{G0}} $
we determine $\bar \nu_G^1$ for \eqref{xmunu} and solve $\bar x_\alpha$ with its law $\bar m^{\rm new}_\alpha$. Denote
 $$(\bar m_\alpha^{\rm new})_{ \alpha\in [0,1]} = \Phi_{ {\bf M}_T^{G0}}( (\bar m_\alpha)_{ \alpha\in [0,1]} ) .$$

If  $h(x,y)$ is a bounded Lipschitz continuous function
with  $|h(x,y)-h(\bar x,\bar y)|\le C_1 |x-\bar x|+C_2(|y-\bar y|\wedge 1)$,  we have
\begin{align}
&\Big|\int h(x, y)g(\alpha, \beta) \nu_\beta^1(t, dy) d\beta-\int h(\bar x, \bar y)g(\alpha, \beta) \nu_\beta^2(t, d\bar y)d\beta\Big| \nonumber  \\
\le & C_1|x-\bar x| + \sup_\beta \Big|\int h(\bar x, y) \nu_\beta^1(t, dy) -\int h(\bar x, \bar y) \nu_\beta^2(t, d \bar y) \Big|  \nonumber   \\
=&  C_1|x-\bar x| + \sup_\beta \Big|\int_{C_T} h(\bar x, X_t(\omega))d m_\beta( \omega) -\int_{C_T} h(\bar x, X_t(\bar \omega) )d\bar m_\beta(\bar \omega) \Big|   \nonumber \\
\le & C_1|x-\bar x|+ C_2\sup_\beta \int_{C_T\times C_T} (|X_t(\omega)-X_t( \bar \omega) |\wedge 1)
d\widehat m_\beta(\omega, \bar \omega),\nonumber
\end{align}
where $X$ is the canonical process, $\omega, \bar \omega\in C_T$, and  $\widehat m_\beta$ is any coupling of $m_\beta$ and $\bar m_\beta$.
Hence
\begin{align}
&|\int h(x, y)g(\alpha, \beta) \nu_\beta^1(t, dy) d\beta-\int h(\bar x, \bar y)g(\alpha, \beta) \nu_\beta^2(t, d\bar y)d\beta| \nonumber  \\
\le \ & C_1|x-\bar x|+C_2 \sup_\beta D_t (m_\beta,\bar m_\beta ).
  \label{hcDt}
\end{align}

By (H2), (H3), the uniform Lipschitz continuity of $\phi_\alpha$ in $x$ by Lemma \ref{lemma:HJBc12} (ii), and \eqref{hcDt}, we obtain
\begin{align*}
&|\widetilde f [x_\alpha, \phi_\alpha (t, x_\alpha|\mu_G),\nu_G^1(t); g_\alpha ]-\widetilde f [\bar x_\alpha, \phi_\alpha (t, \bar x_\alpha|\mu_G), \nu_G^2(t); g_\alpha ] |\\
\le &C_1(|x_\alpha- \bar x_\alpha |\wedge  1) + C_2 \sup_\beta
D_t (m_\beta,\bar m_\beta ) .
\end{align*}
Hence by \eqref{xmunu},
\begin{align}
\sup_{s\le t}|x_\alpha (s)-\bar x_\alpha (s)|&\le
C_1\int_0^t |x_\alpha(s)- \bar x_\alpha(s) |\wedge  1ds \nonumber \\
& + C_3 \int_0^t \sup_\beta |D_s (m_\beta,\bar m_\beta ) |ds. \nonumber
\end{align}
Therefore, by Gronwall's lemma,
\begin{align}
\sup_{s\le t}|x_\alpha (s)-\bar x_\alpha (s)|\wedge 1\le
 C_4\int_0^t \sup_\beta |D_s (m_\beta,\bar m_\beta ) |ds,\nonumber
\end{align}
which combined with the definition of the Wasserstein metric $D_t(\cdot, \cdot)$ implies that
\begin{align}
\sup_\beta |D_t (m^{\rm new}_\beta,\bar m^{\rm new}_\beta ) |\le  C_4\int_0^t \sup_\beta |D_s (m_\beta,\bar m_\beta ) |ds.\label{DtintDs}
\end{align}
By iterating \eqref{DtintDs} as in \cite[p. 174]{szn91}, we can show that for a sufficiently large $k_0$, $ \Phi_{{\bf M}_T^{G0}  }^{k_0} $ is a contraction.
We can further show that $\{\Phi_{{\bf M}_T^{G0} }^k (m_G), k\ge 1\}$ is a Cauchy sequence, and we obtain a unique fixed point 
 $m^*_G$ for  $\Phi_{{\bf M}_T^{G0} }$.
Then we obtain a solution of \eqref{xmunu} by taking
$\nu_\alpha (t)= {\rm Marg}_t(m^*_\alpha)$.
If there are two different solutions with $\nu_G\ne \nu_G'$,
we can  derive  a contradiction by using uniqueness of the fixed point of $\Phi_{{\mathbf M}_T^{G0}}$.
\end{proof}

Now we consider two sets of best response control laws $(\phi_\alpha (t, x_\alpha|\mu_G))_{ \alpha\in [0, 1]} $ and $(\bar \phi_\alpha (t,x_\alpha|\bar \mu_G))_{ \alpha \in [0, 1]}$, where $\mu_G={\rm Marg}(m_G), \bar \mu_G={\rm Marg}(\bar m_G)$ for   $m_G, \bar m_G\in {\bf M}_T^{G1}$
(then clearly $\mu_G, \bar \mu_G \in {\mathcal M}_{[0,T]} $), and use Lemma \ref{lemma:phinu} to solve $(x_\alpha, \nu_G) $ and $(x'_\alpha, \bar \nu_G) $ from the generalized MV-SDEs
\begin{align}
&dx_\alpha =\widetilde f [x_\alpha, \phi_\alpha (t, x_\alpha|\mu_G),\nu_G(t); g_\alpha ] dt +\sigma dw_\alpha(t), \label{xmunu1}\\
&dx_\alpha' =\widetilde f [x_\alpha', \bar \phi_\alpha (t, x'_\alpha|\bar\mu_G),
\bar\nu_G(t); g_\alpha ] dt +\sigma dw_\alpha(t), \label{xmunub}
\end{align}
where $x'_\alpha(0)=x_\alpha (0)$ is given.
Let $m_\alpha^{\rm mv}$ (resp., $\bar m_\alpha^{\rm mv}$) denote the  law of $ x_\alpha$ (resp., $x'_\alpha$).
The following lemma is a generalization of  \cite[Lemma 9]{HMC06}
to the graphon network case.
\begin{lemma} \label{lemma:c2}
For \eqref{xmunu1} and \eqref{xmunub}
there exists a constant $c_2$ independent of $(m_G, \bar m_G)$ such that
\begin{align} 
\sup_\alpha D_T(m_\alpha^{\rm mv}, \bar m_\alpha^{\rm mv}) \le c_2 \sup_{t,x,\alpha}|
\phi_\alpha(t,x|\mu_G(\cdot))-\bar \phi_\alpha(t,x|\bar \mu_G(\cdot))|.
\nonumber
\end{align}
\end{lemma}

\begin{proof} For \eqref{xmunu1}--\eqref{xmunub},
denote
\begin{align*}
\Delta_s = \ & \widetilde f [x_\alpha(s), \phi_\alpha (s, x_\alpha(s)|\mu_G),\nu_G(s); g_\alpha ] -  \widetilde f [x_\alpha'(s), \bar \phi_\alpha (s, x'_\alpha(s)|\bar\mu_G),\bar\nu_G(s); g_\alpha ] .
\end{align*}
We have
\begin{align}
x_\alpha(t)-x_\alpha'(t)=\int_0^t \Delta_sds. \label{xdels}
\end{align}
Noting  $\nu_\alpha(t)={\rm Marg}_t (m_\alpha^{\rm mv})$ and $\bar \nu_\alpha(t)={\rm Marg}_t (\bar m_\alpha^{\rm mv})$,   we have
\begin{align}
|\Delta_s| \le & |\widetilde f [x_\alpha(s), \phi_\alpha (s, x_\alpha(s)|\mu_G),\nu_G(s); g_\alpha ] -  \widetilde f [x_\alpha'(s),  \phi_\alpha (s, x'_\alpha(s)|\mu_G),\bar\nu_G(s); g_\alpha ]| \nonumber \\
&\hskip -0.5cm +| \widetilde f [x_\alpha'(s),  \phi_\alpha (s, x'_\alpha(s)|\mu_G),\bar\nu_G(s); g_\alpha ] -\widetilde f [x_\alpha'(s),  \bar \phi_\alpha (s, x'_\alpha(s)|\bar\mu_G),\bar\nu_G(s); g_\alpha ]|\nonumber \\
\le & C_1|x_\alpha (s)-x_\alpha'(s)  | + C_2 \sup_\beta  D_s (m_\beta^{\rm mv}, \bar m^{\rm mv}_\beta)  \nonumber \\
&+C_3 \sup_{t,x}|
\phi_\alpha(t,x|\mu_G(\cdot))-\bar \phi_\alpha(t,x|\bar \mu_G(\cdot))|, \label{delc12}
\end{align}
where $C_1$, $C_2$ and $C_3$ do not depend on $(\alpha,m_G, \bar m_G)$.
The difference term on the first line is estimated by the method in \eqref{hcDt}.
We have used the fact  that  $\phi_\alpha$ is uniformly Lipschitz in $x$ by Lemma \ref{lemma:unilip} (iii).
Therefore, by \eqref{xdels}--\eqref{delc12},
\begin{align*}
|x_\alpha(t)-x_\alpha'(t)| \le& \int_0^t \Big[C_1|x_\alpha (s)-x_\alpha'(s)  | + C_2\sup_\beta D_s (m^{\rm mv}_\beta, \bar m^{\rm mv}_\beta)\Big] ds \\
&+C_3 t \sup_{t,x}|
\phi_\alpha(t,x|\mu_G(\cdot))-\bar \phi_\alpha(t,x|\bar \mu_G(\cdot))|.
\end{align*}
 By Gronwall's lemma, we obtain
 \begin{align*}
\sup_{0\le s\le t} |x_\alpha(s)-x_\alpha'(s)|\wedge 1 &\le e^{C_1t}C_2 \int_0^t \sup_\beta  D_s (m^{\rm mv}_\beta, \bar m^{\rm mv}_\beta) ds\\
&+ e^{C_1t}C_3t \sup_{t,x}|
\phi_\alpha(t,x|\mu_G(\cdot))-\bar \phi_\alpha(t,x|\bar \mu_G(\cdot))|,
 \end{align*}
which again by the definition of the metric $D_t(\cdot, \cdot)$ leads to
\begin{align}
\sup_\alpha  D_t (m^{\rm mv}_\alpha, \bar m^{\rm mv}_\alpha) \le & e^{C_1t}C_2 \int_0^t \sup_\alpha  D_s (m^{\rm mv}_\alpha, \bar m^{\rm mv}_\alpha) ds  \label{Dtsp} \\
&+  e^{C_1t}C_3 t \sup_{t,x,\alpha}|
\phi_\alpha(t,x|\mu_G(\cdot))-\bar \phi_\alpha(t,x|\bar \mu_G(\cdot))|.\nonumber
\end{align}
The lemma follows from applying Gronwall's lemma to \eqref{Dtsp}.
\end{proof}

\subsection{Existence
Theorem}

We state the main result on the
existence and uniqueness of solutions to the GMFG equation system.
We introduce a contraction condition:

 (H8)  $c_1c_2<1$, where $c_1$ is the constant in the sensitivity condition (H7) and $c_2$ is specified in Lemma \ref{lemma:c2}.

\begin{remark}
By SDE estimates, one can obtain refined bound information on $c_2$. When the  coupling effect is weak or $T$ is small, a small value for $c_2$ can be obtained.
\end{remark}

\begin{remark}  For linear models, a verification of the contraction condition can be done under reasonable model parameters, as in \cite{HMC06}.
\end{remark}

 \begin{theorem} \label{theorem:exist}
Under \emph{(H1)--(H8)},
there exists a unique solution $(V^\alpha, \mu_\alpha(\cdot))_{\alpha \in [0,1]}$ to the GMFG equations \eqref{MFGPDES} and \eqref{clMV}, which \emph{(i)} gives the feedback control  best response \emph{(BR)} strategy
$\varphi(t,x_\alpha| {\mu}_{G}(\cdot); g_\alpha)$
depending only upon the agent's state and the ensemble $\mu_G$ of local mean fields   (i.e. $(x_\alpha, {\mu}_{G})$),  and \emph{(ii)}
generates a Nash equilibrium.
\end{theorem}

\begin{proof}
Step 1 --
We return to the fixed point equation \eqref{mfp0}, which is redisplayed below:
\begin{align}
m_G =\widehat \Gamma \circ {\rm Marg} (m_G),
\label{mfp}
\end{align}
where
$m_G=(m_\alpha)_{\alpha \in [0,1]}\in {\bf M}_T^{G1}  $.  
For $m_G\in {\bf M}_T^{G1}$,  the H\"older continuity in $t$ of the regenerated $\mu_G(\cdot)={\rm Marg}(m_G)$ can be checked by elementary SDE estimates by adapting the proof of \cite[Lemma 7]{HMC06}.

Step 2 -- Take a general $m_G\in  {\bf M}_T^{G1}
 $ to determine $\mu_G={\rm Marg} (m_G)$ and $\phi_\alpha (t,x_\alpha|\mu_G(\cdot)) $.
When $\bar m_G \in  {\bf M}_T^{G1} $ is used, we determine $\bar\mu_G$ and $\bar \phi_\alpha(t,x_\alpha|\bar\mu_G(\cdot))$.
Once the set of strategies
$(\phi_\alpha)_{\alpha\in [0,1]}$ is applied to the generalized MV equation \eqref{xmunu}, by Lemma 3.8, we may
solve  for $(x_\alpha, \nu_G(\cdot))$ such that $x_\alpha$ has the law
$m_\alpha^{\rm new}$ and  ${\rm Marg}_t(m_\alpha^{\rm new})=\nu_\alpha(t)$. This is done in parallel for $\bar m_G$ to generate $\bar m_\alpha^{\rm new}$. We accordingly determine $m_G^{\rm new}$ and $\bar m_G^{\rm new}$.

Step 3 -- By \eqref{RA} and  Lemma \ref{lemma:c2},
 we obtain
\begin{align}
\sup_\alpha D_T(m_\alpha^{\rm new}, \bar m_\alpha^{\rm new})\le
c_1c_2 d(m_G, \bar m_G),\nonumber
\end{align}
which implies
$$
d(m_G^{\rm new}, \bar m_G^{\rm new})\le c_1 c_2 d(m_G, \bar m_G).
$$
Based on the above contraction property, we construct a Cauchy sequence in the complete metric space ${\bf M}_T^G$
by iterating with $m_G$ and establish
existence  of a solution to the GMFG equation system.
To show uniqueness, suppose $m_G$ and $\tilde m_G$ are two fixed points to \eqref{mfp}. We obtain $d(m_G, \tilde m_G)\le c_1 c_2d(m_G,
\tilde m_G)$, which implies $m_G=\tilde m_G$.

The Nash equilibrium property follows from the  best response property of $\phi_\alpha$ for a given vertex $\alpha$.
\end{proof}

\subsection{An Example  on Lipschitz feedback}

The main analysis in Section \ref{sec:gmfg} relies on (H4) to ensure Lipschitz
feedback. We provide a concrete model to
check this assumption.
\begin{example} \label{example:fl}
The dynamics and cost have
\begin{align*}
&f_0(x,u,y)= f_0(x,y) u, \quad f(x,u,y)= f(x,y) u,\\
&l_0(x,u,y)=  l_1(x,y)+{l}_2(x,y)u^2 , \quad l(x,u,y) =l_3(x,y)+ l_4(x,y)u^2,
\end{align*}
where $x,y\in \mathbb{R}$ and  $u\in U=[a,b]$. The functions   $f_0$, $f$, $l_1$, $l_2$, $l_3$, $l_4$ satisfy (H1)--(H3), and there exists $c_0>0$ such that  $l_2,l_4\ge c_0$  for all $x,y$.
\end{example}

Given $\nu_G\in C([0,1],{\mathcal P}_1({\mathbb R}))$, we check
the minimizer of
\begin{align*}
S_\alpha^{\nu_G}(x,q)=\arg \min_{u\in U}\{ q(f_0[x,\nu_\alpha]+f[x,\nu_G; g_\alpha])u + (l_2[x,\nu_\alpha] +l_4[x,\nu_G; g_\alpha])u^2\},
\end{align*}
where $x,q\in \mathbb{R}$.

\begin{proposition} \label{prop:lipu}
Given any compact interval ${\mathcal I}$,
$S_\alpha^{\nu_G}( x,q)$ in Example \ref{example:fl}  is a singleton and Lipschitz continuous in $(x,q)$, where $x\in \mathbb{R}$ and $q\in {\cal I}$, uniformly with respect to $(\nu_G, \alpha)$.
\end{proposition}

\begin{proof}
Consider the function  $\Phi(u)= u^2 -2s u$, where $u\in U$ and   $s$ is a parameter. Its minimum is attained at the unique point
$$
\hat u = \Theta(s)\coloneqq
\begin{cases}
a & \mbox{if}\quad  s\le a,\\
s & \mbox{if} \quad a<s<b, \\
b &\mbox{if} \quad s\ge b.
\end{cases}
$$
Denote the function
$$
h_{\alpha, \nu_G}(x)=-\frac{f_0[x,\mu_\alpha]+f[x,\nu_G; g_\alpha]}{2(l_2[x,\mu_\alpha] +l_4[x,\nu_G; g_\alpha])} .
$$
By elementary estimates we can show
\begin{align}
|h_{\alpha, \nu_G} (x)- h_{\alpha, \nu_G} (y)| \le C_0 |x-y|,\nonumber
\end{align}
where $C_0$ does not depend on $(\nu_G, \alpha)$.
 We have
\begin{align*}
S_\alpha^{\nu_G}(x,q) &= \arg\min_u ( u^2- 2qh_{\alpha, \nu_G}(x)  u)\\
   &= \Theta ({qh_{\alpha, \nu_G}(x)}).
\end{align*}
It is clear that $S_\alpha^{\nu_G}(x,q)$ is a continuous function of $(x,q)$.
For $(x_i,q_i)\in\mathbb{R}\times  {\mathcal I}$, $i=1,2$,
\begin{align}
&|S_\alpha^{\nu_G}(x_1,q_1)-S_\alpha^{\nu_G}(x_2, q_2)| \nonumber\\
&\le  {\rm Lip} (\Theta)|q_1h_{\alpha, \nu_G}(x_1)-q_2h_{\alpha, \nu_G}(x_2) | \nonumber  \\
&\le   {\rm Lip}(\Theta) \Big( |q_1-q_2|  \sup_x| h_{\alpha, \nu_G}(x)| +C_0|x_1-x_2| |q_2|  \Big). \nonumber
\end{align}
In fact, the Lipschitz constant ${\rm Lip}(\Theta)=1$. Note that there exists a fixed constant $C$ such that $ |h_{\alpha, \nu_G}(x)|\le C $ for all $\alpha, \nu_G$.
This proves the proposition.
\end{proof}

If (H1)--(H3) and (H5) hold for Example \ref{example:fl}, they further imply (H4) and (H6) so that the best response  is Lipschitz continuous in $x$
by Lemma \ref{lemma:HJBc12} and Proposition \ref{prop:lipu}.

\section{Performance Analysis}
\label{sec:enash}

In the MFG case it  is  shown \citep{HMC06, CHM17}  that the joint strategy $\{  u^o_{i}(t)  =\varphi_{i}(t,x_{i}(t)|\mu_\cdot), 1\leq i\le  N \} $   yields an $\epsilon$-Nash equilibrium, i.e. for all $\epsilon>0$, there exists $N(\epsilon)$ such that for all $N \geq N(\epsilon)$
\begin{equation}
\label{EPSILONNASHINEQ}
J_i^N({u}_i^\circ, {u}_{-i}^\circ)-\epsilon \leq\inf_{u_i \in\mathcal{U}_{i} } J_i^N(u_i, {u}_{-i}^\circ) \leq J_i^N({u}_i^\circ, {u}_{-i}^\circ).
\end{equation}
This form of approximate Nash equilibrium is a principal result of the MFG analyses in the sequence  \cite{HMC06,CHM17,SEC16} and  in many other  studies. The importance of  \eqref{EPSILONNASHINEQ} is that it states that the cost function of any agent in a finite population can be reduced by at most  $\epsilon$ {\it if it changes unilaterally  from the infinite population  MFG  feedback law} while all other agents remain with  the infinite population based control strategies. The main result of this section is that the same property holds for GMFG systems.

 Throughout this section, let  $\mu_G(\cdot)$ be solved from the GMFG equations \eqref{MFGPDES} and \eqref{clMV}.

\subsection{The $\epsilon$-Nash Equilibrium}
The analysis of GMFG systems as limits of finite objects  necessarily involves the consideration of  graph limits and double limits  in population and graph order. A corresponding set of assumptions  is given below.

(H9)  $M_k\to \infty$ and $\min_{1\le l\le M_k}|\maC_l|\to \infty$ as $k\to \infty$.

(H10) All agents have i.i.d. initial states with distribution $\mu_0^x$ and $E|x_i(0)|\le C_0$.

\begin{remark}
 (H10) is a simplifying assumption to keep further notation light. It may be generalized to $\alpha$ dependent initial distributions.
\end{remark}

(H11) The sequence  $\{G_k; 1\le k< \infty\}$  and the graphon limit satisfy
$$
\lim_{k\to \infty}\max_i\sum_{j=1}^{M_k} \Big|\frac{1}{M_k} g_{\maC_i \maC_j}^k-
\int_{\beta\in I_j} g_{I_i^*,\beta}
d\beta\Big|=0,
$$
where
 $I^*_i$ is the midpoint of the subinterval $I_i\in \{I_1,\ldots, I_{M_k}\}$ of length $1/{M_k}$.

\begin{remark}
Assumption (H11) specifies the nature of  the approximation error between $g^k$ for the finite graph  and the graphon function $g$.
\end{remark}

The next proposition shows that under (H5) and (H11), the limit $g$ is well determined.

\begin{proposition}
For the given sequence $\{g^k, k\ge 1\}$ under \emph{(H9)}, if there exists a graphon $g$ satisfying \emph{(H5)} and \emph{(H11)}, then it is unique.
\end{proposition}

\begin{proof}
Assume there is another graphon $\hat g$ satisfying (H5) and (H11).
Fix any $\epsilon>0$ and any  $ \mathcal{S}\times \mathcal{T} \subset [0,1]\times [0,1] $. By Lemma \ref{lemma:ST1},  there exists a sufficiently large  $k_0$ (depending on $\epsilon$, ${\mathcal S}$ and ${\mathcal T}$), such that for both $g$ and $\hat g$ we have
\begin{align}
\Big|\int_{\mathcal{S}\times \mathcal{T}} ( g^{k^0}-g )dxdy \Big|\le \epsilon,\qquad
\Big|\int_{\mathcal{S}\times \mathcal{T} } ( g^{k^0}-\hat g )dxdy\Big| \le \epsilon.\nonumber
\end{align}
Hence
\begin{align}
\Big|\int_{\mathcal{S}\times \mathcal{T}  } ( g-\hat g )dxdy\Big| \le 2\epsilon. \nonumber
\end{align}
Since $\mathcal{S}\times \mathcal{T}  $ is arbitrary, we have
$\| g-\hat g\|_\Box \le 2\epsilon.$
Since $\epsilon$ is arbitrary, we have
$\| g-\hat g\|_\Box=0.$
But the  cut norm is a norm, so we have $g=\hat g$.
\end{proof}

For the $\epsilon$-Nash equilibrium analysis, we consider  a sequence of games each defined on a  finite graph $G_k$. Recall that
there is a total of
$N=\sum_{l=1}^{M_k} |\maC_l|$ agents.

Suppose the cluster $\maC(i)$ of agent ${\cal A}_i$ corresponds to the subinterval $I(i)\in \{I_1, \ldots, I_{M_k}\}$. The
agent ${\cal A}_i$ takes the midpoint $I^*(i)$ of the subinterval $I(i)$ and uses the GMFG equations to determine its control law
  \begin{align}
\hat u_i=\varphi(t, x_i|\mu_G(\cdot); g_{I^*(i)}), \quad 1\le i\le N, \label{hatuIi}
\end{align}
 which we simply write as $ \varphi(t,x_i, g_{I^*(i)})$.
Denote the resulting state process by
$\hat x_i$, $1\le i\le N$.
Recall that
\begin{align*}
&f_0(x_i^N, u_i^N, {\mathcal C}(i))=     \frac{1}{|\maC(i)|}
\sum_{j\in \maC(i)}f(x_i^N, u_i^N, x_j^N) , \\
&f_{G_k} (x_i^N, u_i^N, g_{\maC(i)}^k)  =   \frac{1}{M_k}
\sum_{l=1}^{M_k} g^k_{\maC(i)\maC_l} \frac{1}{|\maC_l|}
\sum_{j\in \maC_l}f(x_i^N, u_i^N, x_j^N) ,
\end{align*}
where the superscript $N$ is added to indicate the population size.
The closed-loop system of  $N$ agents on the finite graph $G_k$
 under the set of strategies \eqref{hatuIi} is given by
\begin{align}
\mbox{\it System A:}\quad d\hat x_i^N =& f_0(\hat x_i^N, \varphi(t, \hat x_i^N, g_{I^*(i)}), {\mathcal C}(i)  )dt  \nonumber \\
&+f_{G_k}(\hat x_i^N, \varphi(t, \hat x_i^N, g_{I^*(i)}) , g_{\maC(i)}^k)dt
+\sigma dw_i ,  \label{xNdeccontrol}
\end{align}
where $ 1\le i\le N$ and $\hat x_i^N(0)=x_i^N(0)$. Note that $g_{\maC(i)}^k$ appears in $f_{G_k}$ as  determined by the finite population system dynamics.
 We  state the following main result.
\begin{theorem} \label{theorem:eNash}
\emph{({\bf $\epsilon$-Nash equilibrium})}
Assume \emph{(H1)--(H11)} hold.
Then when the strategies \eqref{hatuIi} determined by the
GMFG equations \eqref{MFGPDES} and \eqref{clMV} are applied to
a sequence of finite graph systems $\{G_k; 1\le k<\infty\}$, the $\epsilon$-Nash equilibrium property holds
where $\epsilon\to 0$ as $k\to \infty$, and where the unilateral agent ${\mathcal A}_i$ uses a centralized Lipschitz feedback strategy $\psi(t,x_i,x_{-i})$, where $x_{-i}$ denotes  the set of states of all other agents.
\end{theorem}

We first explain the basic idea for
the demonstration of the $\epsilon$-Nash equilibrium property.
Suppose all other players, except agent ${\cal A}_{\cki}$,  employ the control strategies based on the GMFG equation system.
When ${\cal A}_{\cki}$ employs a different strategy,
the resulting change in its performance can be measured using
a limiting stochastic control problem where
both the system dynamics and the cost are subject to small perturbation due to the mean field approximation of the effects of all other  agents. The proof is  technical and preceded by some lemmas.

\subsection{Proof of Theorem \ref{theorem:eNash}}

Suppose $x_{\cki}^N$ is determined from a general feedback control law
$u_{\cki}^N$ instead of the GMFG best response. With the exception of  agent ${\mathcal A}_{\cki}$ with its unilateral strategy, all other agents ${\cal A}_j$, $j\ne {\cki}$, still have strategies determined by \eqref{hatuIi}.
We introduce the  system:
\begin{align}
\mbox{\it System B:}\quad\begin{cases}
dx_{\cki}^N = f_0(x_{\cki}^N, u_{\cki}^N, \maC(\cki)  )
dt+ f_{G_k}(x_{\cki}^N,  u_{\cki}^N, g_{\maC(\cki)}^k)dt +\sigma
dw_{\cki} ,\\ 
dx_j^N = f_0(x_j^N,\varphi(t, x_j^N, g_{I^*(j)}),\maC(j)   )
dt \\
\qquad\qquad + f_{G_k}(x_j^N,  \varphi(t, x_j^N, g_{I^*(j)}), g_{\maC(j)}^k)dt +\sigma dw_j , \\
\qquad\qquad   j\ne {\cki},\quad  1\le j\le N.
\end{cases}
\label{uijne}
\end{align}
We note that $x_j^N$ is affected by the unilateral choice of strategy by  ${\mathcal A}_{\cki}$ due to the coupling in
$f_0$ and  $f_{G_k}$. For this reason, $x_j^N$ differs from $\hat x_j^N$ in \eqref{xNdeccontrol} although   the control law of ${\cal A}_j$, $j\ne \cki$, remains the same. The central task is to estimate by how much  ${\cal A}_{\cki}$ can reduce its cost.

To  facilitate the performance estimate in System $B$, we introduce two auxiliary systems below.  Consider
\begin{align}
\mbox{\it System C:}\quad
dy_i^N=\ &\int_{\mathbb R} f_0(y_i^N, \varphi(t, y_i^N, g_{I^*(i)}), z)
m_{y_i^N} (dz) dt \nonumber\\
&+  \frac{1}{M_k}\sum_{l=1}^{M_k}
g^k_{\maC(i) \maC_l} \frac{1}{|\maC_l|}
\sum_{j\in \maC_l} \int_\mathbb{R} f(y_i^N, \varphi(t,y_i^N, g_{I^*(i)}) , z )m_{y_j^N}(dz) dt \nonumber \\
&+\sigma dw_i \nonumber \\
=\ &\int_{\mathbb R}f_0(y_i^N, \varphi(t, y_i^N, g_{I^*(i)}), z)m_{y_i^N} (dz) dt \nonumber\\
&+ \frac{1}{M_k}\sum_{l=1}^{M_k}
g^k_{\maC(i) \maC_l}
 \int f(y_i^N, \varphi(t,y_i^N, g_{I^*(i)}) , z )m_l^N(t,dz) dt \nonumber \\
&+\sigma dw_i, \label{gmkv}
\end{align}
where $1\le i\le N$ and $y_i^N(0) =x_i^N(0)$, and $m_{y_j^N(t)}$  denotes the law of $y_j^N(t)$.  Each Brownian motion  $w_i$ is the same as in
\eqref{xNdeccontrol}.
The second equality holds since all processes in cluster $\maC_l$ have the same distribution denoted by $m_l^N(t,dz)$ at time $t$.
  It is clear that the  processes $y_1^N, \ldots, y_N^N$ are independent, and $\{y_j^N, j\in \maC_l\}$ are i.i.d. for any given $l$.

Next we introduce
\begin{align}
\mbox{\it System D:}\quad dy_i^\infty(t) &= \widetilde f[y_i^\infty(t), \varphi(t, y_i^\infty(t), g_{I^*(i)}), \mu_G(t); g_{I^*(i)\ }] dt +\sigma dw_i(t),   \label{ygmfg}
\end{align}
where $1\le i\le N$ and $y_i^\infty(0)= x_i^N(0)$.
Here $w_i$ is the same as in \eqref{xNdeccontrol}.
The process $y_i^\infty  $ is generated by the closed-loop dynamics for an agent at the node $I^*(i)$ associated with the cluster $\maC(i)$ using the GMFG based control law \eqref{hatuIi}  while situated in  an infinite population represented by the ensemble $\mu_G(\cdot)$ of local mean fields. We  view \eqref{ygmfg} as an instance of the generic equation \eqref{MVNLMinorDyn2} under the control law \eqref{hatuIi}.  By Theorem \ref{theorem:exist}, $y^\infty_i(t)$ has the law $\mu_{I^*(i)}(t)$. Note that if $j\in \maC(i)$, $y_i^\infty$ and $y_{j}^\infty$ are two processes of  the same distribution.

We shall denote the $A$ to $C$ system deviation by  $\epsilon_{1,N}$, the $C$ to $D$ deviation by  $\epsilon_{2,N}$ and the (non-unilateral agent) $B$ to $D$ deviation by  $\epsilon_{3,N}$.
Specifically, we set
\begin{align*}
&\epsilon_{1,N}=\sup_{i\le N, t}E |\hat x_{i}^N(t) - y_i^N(t) |,\qquad \epsilon_{2,N}=\sup_{i\le N,t}E|y_i^N(t)- y_i^\infty(t)|,\\
&\epsilon_{3,N}=\sup_{u_{\cki}^N,t,\cki\ne j\le N}E|x_j^N(t)-y_j^\infty(t)|,
\end{align*}
where $x_j^N$ is given by \eqref{uijne}.

\begin{lemma}
 The SDE system \eqref{gmkv} has a unique solution
 $(y_1^N,\ldots,  y_N^N)$.
\end{lemma}
\proof The proof is similar to \cite[Theorem 6]{HMC06}.
\endproof

\begin{lemma} \label{lemma:xNyN}
 $\epsilon_{1,N}\to 0$ as $N\to \infty$ (due to $k\to \infty$). 
\end{lemma}

\begin{proof}
We write
\begin{align}
\hat x_i^N(t)-y_i^N(t)=& \int_0^t \frac{1}{|\maC(i)|} \sum_{j\in \maC(i)}\xi_{ij}^0(s) ds \label{hximinusyi} \\
&+\int_0^t \frac{1}{M_k} \sum_{l=1}^{M_k} g_{\maC(i) \maC_l}^k \frac{1}{|\maC_l|}\sum_{j\in \maC_l}  \xi_{ij}(s) ds, \nonumber
\end{align}
where
\begin{align*}
&\xi_{ij}^0(s)=f_0(\hat x_i^N,\varphi(s, \hat x_i^N, g_{I^*(i)}), \hat x_j^N)-
\int_{\mathbb R}f_0( y_i^N,\varphi(s,  y_i^N, g_{I^*(i)}), z)m_{y_j^N(s)}(dz),   \\
&\xi_{ij}(s)= f(\hat x_i^N,\varphi(s, \hat x_i^N, g_{I^*(i)}), \hat x_j^N) - \int_{\mathbb{R}}f (y_i^N,\varphi(s,  y_i^N, g_{I^*(i)} ), z) m_{y_j^N(s)} (dz).
\end{align*}
We check  the second line of \eqref{hximinusyi} first.  Write
\begin{align*}
\xi_{ij}(s)=& f(\hat x_i^N,\varphi(s, \hat x_i^N, g_{I^*(i)}),\hat x_j^{N})
-f( y_i^N,\varphi(s,  y_i^N, g_{I^*(i)}), y_j^N)   \\
&+f( y_i^N,\varphi(s,  y_i^N, g_{I^*(i)}), y_j^N)- \int_{\mathbb{R}}f (y_i^N,\varphi(s,  y_i^N, g_{I^*(i)} ), z) m_{y_j^N(s)} (dz)  .
\end{align*}
Denote
$$
\zeta_{ij}=f( y_i^N,\varphi(s,  y_i^N, g_{I^*(i)}), y_j^N)- \int_{\mathbb{R}}f (y_i^N,\varphi(s,  y_i^N, g_{I^*(i)} ), z) m_{y_j^N(s)} (dz).
$$
By the Lipschitz conditions (H2), (H3) and the best response's uniform Lipschitz continuity in $x$ by Lemma \ref{lemma:unilip}, we obtain
\begin{align*}
&\Big|\frac{1}{M_k} \sum_{l=1}^{M_k} g_{\maC(i) \maC_l}^k \frac{1}{|\maC_l|}\sum_{j\in \maC_l}  \xi_{ij}(s)\Big|\\
 \le& C|\hat x_i^N-y_i^N| +  \frac{C}{M_k} \sum_{l=1}^{M_k} g_{\maC(i) \maC_l}^k \frac{1}{|\maC_l|}\sum_{j\in \maC_l} |\hat x_j^N-y_j^N|\\
&+\Big|\frac{1}{M_k} \sum_{l=1}^{M_k} g_{\maC(i) \maC_l}^k \frac{1}{|\maC_l|}\sum_{j\in \maC_l} \zeta_{ij} \Big|.
\end{align*}
Then by independence of $y_i^N$, $1\le i\le N$,
\begin{align*}
E\Big|\frac{1}{M_k} \sum_{l=1}^{M_k} g_{\maC(i) \maC_l}^k \frac{1}{|\maC_l|}\sum_{j\in \maC_l} \zeta_{ij}\Big|^2&\le C\sum_{l=1}^{M_k}\sum_{j\in \maC_l} \frac{|g_{\maC(i) \maC_l}^k|^2}{M_k^2 |\maC_l|^2}\\
&\le \frac{C}{M_k\min_l|\maC_l|}.
\end{align*}
The estimate for   $\frac{1}{|\maC(i)|} \sum_{j\in \maC(i)}\xi_{ij}^0(s) $ can be obtained similarly.
Now it follows from  \eqref{hximinusyi} that
\begin{align}
&E|\hat x_i^N(t)-y_i^N(t)|\le C\int_0^t E|\hat x_i^N(s)-y_i^N(s)|ds \nonumber\\
&+ \frac{C}{M_k} \sum_{l=1}^{M_k}  \frac{g_{\maC(i) \maC_l}^k}{|\maC_l|}\sum_{j\in \maC_l}\int_0^t E|\hat x_j^N(s)-y_j^N(s)|ds\nonumber \\
&+\frac{C}{|\maC(i)|}\sum_{j\in \maC(i)}\int_0^t E|\hat x_j^N(s)-y_j^N(s)| ds +\frac{C_1}{\sqrt{M_k\min_l|\maC_l|}}+\frac{C}{\sqrt{|\maC(i)|}} \nonumber  \\
&\le C_2\int_0^t \Delta^N(s) ds+\frac{C_3}{\sqrt{\min_l|\maC_l|}} ,\nonumber
\end{align}
where
$
\Delta^N(t)=\max_{1\le i\le N}E|\hat x_i^N(t)-y_i^N(t)|.
$
The above further implies
\begin{align}
\Delta^N(t) \le C_2\int_0^t \Delta^N(s) ds+\frac{C_3}{\sqrt{\min_l|\maC_l|}}.\nonumber
\end{align}
The lemma follows from (H9) and Gronwall's lemma.
\end{proof}

\begin{lemma}\label{lemma:yNyinf}
We have $\epsilon_{2,N}\to 0$ as $N\to \infty$. 
\end{lemma}

\begin{proof}
For System $D$ and $1\le i\le N$, $y_i^\infty(t)$ has the law $\mu_{I^*(i)} (t) $ and we write
\begin{align}
dy_i^\infty =\ & \int_{\mathbb R}f_0(y_i^\infty, \varphi(t, y_i^\infty, g_{I^*(i)} ), z)\mu_{I^*(i)} (t,dz) dt +\sigma dw_i \label{yiinfphi}    \\
&+ \int_0^1 \int_{\mathbb{R}} f(y_i^\infty,\varphi(t, y_i^\infty, g_{I^*(i)}  )  , z) g(I^*(i), \beta)\mu_\beta (t,dz) d\beta \ dt. \nonumber
\end{align}
Set
\begin{align*}
&\int_0^1 \int_{\mathbb{R}} f(y_i^\infty,\varphi(t, y_i^\infty, g_{I^*(i)}  )  , z) g(I^*(i), \beta)\mu_\beta (t,dz) d\beta \\
=\ &\sum_{l=1}^{M_k} \int_{\beta\in I_l}\int_{\mathbb{R}}f(y_i^\infty,\varphi(t, y_i^\infty, g_{I^*(i)}  )  , z) g(I^*(i), \beta)\mu_\beta (t,dz) d\beta \\
 \eqqcolon  &\ \xi_k^i
 + \zeta_k^i,
\end{align*}
where
\begin{align}
&\xi_k^i=\sum_{l=1}^{M_k} \int_{\beta\in I_l} g(I^*(i), \beta) d\beta \int_{\mathbb{R}}f(y_i^\infty,\varphi(t, y_i^\infty, g_{I^*(i)}  )  , z)\mu_{I_l^*} (t,dz),\nonumber\\
& \zeta_k^i=\sum_{l=1}^{M_k} \zeta^i_{kl}, \nonumber  \\
&\zeta_{kl}^i \coloneqq  \int_{\beta\in I_l}\int_{\mathbb{R}}f(y_i^\infty,\varphi(t, y_i^\infty, g_{I^*(i)}  )  , z) g(I^*(i), \beta)[\mu_\beta(t,dz)-\mu_{I_l^*} (t,dz)] d\beta.\label{zetaldb}
\end{align}
We rewrite
\begin{align*}
\xi_k^i = &\sum_{l=1}^{M_k} \frac{g^k_{\maC(i)\maC_l}}{M_k}  \int_{\mathbb{R}}f(y_i^\infty,\varphi(t, y_i^\infty, g_{I^*(i)}  )  , z)\mu_{I_l^*} (t,dz) \\
 &+\sum_{l=1}^{M_k}  \left[ \int_{\beta\in I_l} g(I^*(i), \beta) d\beta -\frac{g^k_{\maC(i)\maC_l}}{M_k} \right]\int_{\mathbb{R}}f(y_i^\infty,\varphi(t, y_i^\infty, g_{I^*(i)}  )  , z)\mu_{I_l^*} (t,dz) \\
 \eqqcolon\ & \xi^i_{k,1}+\xi^i_{k,2}.
\end{align*}
By (H11) and boundedness of $f$,  we have $\lim_{k\to \infty}\sup_{t,\omega}\max_{1\le i\le N}|\xi^i_{k,2}| =0   $ so that
\begin{align}
\lim_{k\to \infty}\max_i\int_0^TE|\xi^i_{k,2}(t)|dt =0. \label{xik2to0}
\end{align}
Now \eqref{yiinfphi} may be rewritten in the form
\begin{align}
dy_i^\infty=\ &\int_{\mathbb R} f_0(y_i^\infty, \varphi(t, y_i^\infty, g_{I^*(i)} ),z)\mu_{I^*(i)}(t,dz)dt +\sigma dw_i \nonumber \\
&+ (\xi^i_{k,1} +\xi^i_{k,2}+ \zeta^i_k)dt. \nonumber
\end{align}

In view of    \eqref{gmkv},
we have
\begin{align*}
&y_i^\infty(t) -y_i^N(t)\\
=& \int_0^t
\int_{\mathbb R}[f_0(y_i^\infty, \varphi(s, y_i^\infty, g_{I^*(i)} ),z)\mu_{I^*(i)}(s,dz) - f_0(y_i^N, \varphi(s, y_i^N, g_{I^*(i)}),z)m_{y^N_i(s)}(dz) ]  ds \\
  & +\frac{1}{M_k}\sum_{l=1}^{M_k}  g^k_{\maC(i)\maC_l} \int_0^t\int_{\mathbb{R}}f(y_i^\infty,\varphi(s, y_i^\infty, g_{I^*(i)}  )  , z)\mu_{I_l^*} (s,dz)ds\\
&-  \frac{1}{M_k}\sum_{l=1}^{M_k}
g^k_{\maC(i) \maC_l}\int_0^t
 \int_{\mathbb{R}} f(y_i^N, \varphi(s,y_i^N, g_{I^*(i)}) , z )m_l^N(s,dz) ds\\
&+ \int_0^t (\xi^i_{k,2}+ \zeta^i_k)ds.
\end{align*}
Denote
\begin{align}
\Delta_{il}(s)=&\Big|\int_{\mathbb{R}}f(y_i^\infty,\varphi(s, y_i^\infty, g_{I^*(i)}  )  , z)\mu_{I_l^*} (s,dz) \nonumber  \\
& -\int_{\mathbb{R}} f(y_i^N, \varphi(s,y_i^N, g_{I^*(i)}) , z )m_l^N(s,dz)\Big|  . \nonumber
\end{align}
It follows that
\begin{align*}
\Delta_{il}(s)\le &\Big|\int_{\mathbb{R}}f(y_i^\infty,\varphi(s, y_i^\infty, g_{I^*(i)}  )  , z)\mu_{I_l^*} (s,dz)\\
& -\int f(y_i^N, \varphi(s,y_i^N, g_{I^*(i)}) , z )\mu_{I_l^*}(s,dz)  \Big|\\
&+ \Big|\int_{\mathbb{R}} f(y_i^N, \varphi(s,y_i^N, g_{I^*(i)}) , z )\mu_{I_l^*}(s,dz) \\
&- \int_{\mathbb{R}} f(y_i^N, \varphi(s,y_i^N, g_{I^*(i)}) , z )m_l^N(s,dz) \Big| \\
\eqqcolon & \ \Delta_{il1}(s) +\Delta_{il2}(s).
\end{align*}
By the Lipschitz condition (H2), for any fixed $y\in \mathbb{R}$, we have
\begin{align*}
&\Big |\int_{\mathbb{R}} f(y, \varphi(s,y, g_{I^*(i)}) , z )\mu_{I_l^*}(s,dz) - \int_{\mathbb{R}} f(y, \varphi(s,y, g_{I^*(i)}) , z )m_l^N(s,dz)\Big|\\
=&| E f(y, \varphi(s,y, g_{I^*(i)}) , y_j^\infty ) - E f(y, \varphi(s,y, g_{I^*(i)}) , y_j^N )|\\
\le & CE|y_j^\infty(s) -y_j^N(s) |,
\end{align*}
where $j\in \maC_l$ and we have used the fact that $y_i^\infty(t)$ in \eqref{yiinfphi} has the law $\mu_{I^*(i)}(t)$ and that $y_j^N(t) $ has the law $m_l^N(t)$.      Consequently, we have
for $j\in \maC_l$, with probability one,
\begin{align}
&\Delta_{il2}(s)
  \le  CE|y_j^\infty(s) -y_j^N (s)|. \label{mumyy}
\end{align}
We  estimate $\Delta_{kl1}$ using the Lipschitz property of $f$ and $\varphi_{I^*(i)}$.  Now it follows that
\begin{align}
E\Delta_{il}(s)\le C E|y_i^\infty(s) -y_i^N (s) | +  CE|y_j^\infty(s) -y_j^N (s) |, \quad j\in \maC_l .\nonumber
\end{align}
We similarly estimate the difference term involving $f_0$. Therefore,
\begin{align*}
E|y_i^\infty(t) -y_i^N(t)|\le \ &
C\int_0^t E|y_i^\infty -y_i^N|ds  +\int_0^t
E(|\xi^i_{k,2}| + |\zeta^i_k|)ds\\
&+    \frac{1}{M_k}\sum_{l=1}^{M_k}
g^k_{\maC(i) \maC_l}\int_0^t E\Delta_{il}ds\\
\le & C_1\int_0^t \max_i E|y_i^\infty -y_i^N|ds  +\int_0^t
E(|\xi^i_{k,2}| + |\zeta^i_k|)ds\\
&+    \frac{C}{M_k}\sum_{l=1}^{M_k}
g^k_{\maC(i) \maC_l}\int_0^t \max_jE|y_j^\infty-y_j^N |ds \\
\le  & 2 C_2\int_0^t \max_i E|y_i^\infty -y_i^N|ds  +\int_0^t
E(|\xi^i_{k,2}| + |\zeta^i_k|)ds.
\end{align*}
Consequently,
\begin{align*}
\max_iE|y_i^\infty(t) -y_i^N(t)|\le \ &    2 C_2\int_0^t \max_i E|y_i^\infty -y_i^N|ds  +\max_i\int_0^t
E(|\xi^i_{k,2}| + |\zeta^i_k|)ds.
\end{align*}
By Gronwall's lemma,
\begin{align}
\sup_{0\le t\le T}\max_iE|y_i^\infty(t) -y_i^N(t)|\le C\max_i \int_0^T
E(|\xi^i_{k,2}| + |\zeta^i_k|)ds. \label{yiyinfxi}
\end{align}

To estimate \eqref{zetaldb}, by (H2) we derive
\begin{align*}
\zeta^i_{kl,\beta} \coloneqq & \Big|\int_{\mathbb{R}}f(y_i^\infty,\varphi(t, y_i^\infty, g_{I^*(i)}  )  , z) [\mu_\beta(t,dz)-\mu_{I_l^*} (t,dz)] \Big|\\
 =&\Big|\int_{\mathbb{R}^2}[ f(y_i^\infty,\varphi(t, y_i^\infty, g_{I^*(i)}  )  , z_1) -  f(y_i^\infty,\varphi(t, y_i^\infty, g_{I^*(i)}  )  , z_2)]
\widehat\gamma(dz_1, dz_2) \Big|\\
\le & C \int_{\mathbb{R}^2} |z_1-z_2 | \widehat\gamma(dz_1, dz_2),
\end{align*}
where the probability measure $\widehat\gamma$ is any coupling of $\mu_\beta(t)$ and $\mu_{I_l^*}(t)$ and $C$ is the Lipschitz constant of $f$. Since the coupling $\widehat \gamma$ is arbitrary, we have $\zeta^i_{kl,\beta}\le C W_1(\mu_\beta(t), \mu_{I^*(i)}(t))$.
Denote
$\delta_k^\mu=\sup_{l\le M_k}\sup_{\beta\in I_l,t\le T}
W_1(\mu_\beta(t),
\mu_{I_l^*}(t)). $
Then with probability one,
\begin{align}
|\zeta^i_{kl}(t)| \le C \delta^\mu_k/M_k \nonumber
\end{align}
in view of \eqref{zetaldb},
and therefore $\max_i|\zeta^i_k(t)|\le C\delta_k^\mu$. Note that $\delta^\mu_k\to 0$ as $k\to \infty$ by Lemma \ref{lemma:wasscnt}.
Recalling \eqref{xik2to0},  the right hand side of \eqref{yiyinfxi} tends to 0 as $k\to \infty$.
This completes the proof.
\end{proof}

\begin{lemma} \label{lemma:xhyinf}
$ \lim_{N\to \infty}\sup_{t,i\le N} E|\hat x_i^N- y_i^\infty|=0$.
\end{lemma}
\proof
The lemma follows from Lemmas \ref{lemma:xNyN} and \ref{lemma:yNyinf}.
\endproof

\begin{lemma}\label{lemma:e3}
$\lim_{N\to \infty} \epsilon_{3,N}=0.
$
\end{lemma}

\proof For $(\hat x_1^N,\ldots, \hat x_N^N )$ in System $A$ and
   $(x_1^N,\ldots,  x_N^N  )$ in System $B$,  we compare the SDEs of $\hat x_j^N$ and $x_j^N$ and apply Gronwall's lemma to obtain
$$
\sup_{u_{\cki}^N,t, j\ne \cki}|x_j^N-\hat x_j^N|\le  \frac{C}{{\min_l |\maC_l|}}.
$$
  Next by Lemma \ref{lemma:xhyinf}, we obtain the desired estimate. \endproof

Consider the limiting optimal control problem with dynamics  and cost
\begin{align}
&dx_{\cki}^\infty = \widetilde f[x_{\cki}^\infty, u_{\cki}, \mu_G; g_{I^*({\cki}) }] dt +\sigma dw_{\cki}, \label{yiofD}\\
&J_{\cki}^*= E\int_0^T\widetilde l [x_{\cki}^\infty, u_{\cki}, \mu_G; g_{I^*({\cki})}]dt ,\label{jistar}
\end{align}
where $x_{\cki}^\infty (0)= x_{\cki}^N(0)$ and $\mu_G(\cdot)$ is given by the GMFG equation system.

To establish the $\epsilon$-Nash equilibrium property, the cost of agent ${\cal A}_{\cki}$ within the $N$ agents can be written using the mean field limit dynamics and cost, both involving $\mu_G(\cdot)$, up to a small error term that can be bounded uniformly with respect to $u_{\cki}^N$, while ${\cal A}_{\cki}$ chooses its control $u_{\cki}^N$.
It can further have little improvement due to the best response property of $\varphi(t, x_{\cki}|\mu_G(\cdot); g_{I^*({\cki})})$ within the mean field limit.
We rewrite the first  equation in \eqref{uijne} of System $B$ as
\begin{align}
dx_{\cki}^N = \widetilde f[x_{\cki}^N, u_{\cki}^N, \mu_G; g_{I^*({\cki})}] dt
+ (\delta_{f_0}^k(t)+\delta_f^k(t)) dt  +\sigma dw_{\cki},\label{xindel}
\end{align}
where $\delta_{f_0}^k=f_0(x_{\cki}^N, u_{\cki}^N, {\maC({\cki})})- f_0[x_{\cki}^N, u_{\cki}^N, \mu_{I^*({\cki})}]  $  and  $\delta_f^k=f_{G_k}(x_{\cki}^N, u_{\cki}^N, g_{\maC({\cki})}^k)- f[x_{\cki}^N, u_{\cki}^N, \mu_G; g_{I^*({\cki}) }] $.  Similarly the cost of ${\mathcal A}_{\cki}$ in System $B$ is written as
\begin{align}
J_{\cki}^N(u_{\cki}^N)= E\int_0^T (\widetilde l[x_{\cki}^N, u_{\cki}^N, \mu_G; g_{I^*({\cki})}]
+\delta_{l_0}^k(t)+ \delta_l^k (t))dt, \nonumber
\end{align}
where we have
$\delta_{l_0}^k= l_{0}(x_{\cki}^N, u_{\cki}^N, {\maC({\cki})})- l_0[x_{\cki}^N, u_{\cki}^N, \mu_{I^*({\cki})}]  $ and $\delta_l^k= l_{G_k}(x_{\cki}^N, u_{\cki}^N, g_{\maC({\cki})}^k)- l[x_{\cki}^N, u_{\cki}^N, \mu_G; g_{I^*({\cki}) }] $.
Note that all other agents have applied the control laws $\varphi(t,x_j^N, g_{I^*(j)})$, $j\ne {\cki}$. So we only indicate $u_{\cki}^N$ within $J_{\cki}^N$.
It is clear that  $\delta_{f_0}^k$,  $\delta_f^k$, $\delta_{l_0}^k$, and  $\delta_l^k$ are all affected by the control law $u_{\cki}^N$. Let $ {\mathbold y}^\infty_t =(y_1^\infty(t), \ldots, y_N^\infty(t) )$ for System $D$.  Our next step is to derive a uniform upper bounded for
$E|\delta_f^k|$ and $E|\delta_l^k|$ with respect to $u_{\cki}^N$.

Define the two random variables
\begin{align*}
&\Delta_f^k(z, u, {\mathbold y}^\infty_t )= \frac{1}{M_k}\sum_{l=1}^{M_k}g^k_{{\mathcal C}({\cki}) {\mathcal C}_l}
\frac{1}{|{\mathcal C}_l|}\sum_{j\in {\mathcal C}_l}f(z,u, y_j^\infty(t))
 - f[z, u, \mu_G(t); g_{I^*({\cki}) }],\\
&\Delta_l^k(z, u,{\mathbold y}^\infty_t  )= \frac{1}{M_k}\sum_{l=1}^{M_k}g^k_{{\mathcal C}({\cki}) {\mathcal C}_l}
\frac{1}{|{\mathcal C}_l|}\sum_{j\in {\mathcal C}_l}l(z,u, y_j^\infty(t))  -
{ l}[z, u, \mu_G(t); g_{I^*({\cki}) }],
\end{align*}
where $z\in \mathbb{R}$ and $u\in U$ are deterministic and  fixed.
\begin{lemma}\label{lemma:DD2}
We have
\begin{align}
\lim_{k\to \infty}\sup_{z,u,t}E (|\Delta_f^k(z, u,{\mathbold y}^\infty_t)|^2+|\Delta_l^k(z, u,{\mathbold y}^\infty_t )|^2)=0.\label{Edd}
\end{align}
\end{lemma}
\begin{proof}
As in the proof of Lemma  \ref{lemma:yNyinf}, we approximate $\mu_\beta$, $\beta\in [0,1]$, by using a finite number of points of $\beta$, and next expand the two quadratic terms in \eqref{Edd}. The estimate is carried out using (H11) and  Lemma \ref{lemma:wasscnt}.
\end{proof}

\begin{lemma}\label{lemma:PDel}
For any given constant $C_z>0$ and any $\epsilon\in (0,1)$,
\begin{align*}
&\lim_{k\to \infty}\inf_tP(\cap_{(z,u)\in [-C_z, C_z]\times U}\{ |\Delta_f^k(z,u, {\mathbold y}^\infty_t)|\le \epsilon\} )=1,\\
& \lim_{k\to \infty}\inf_t P(\cap_{(z,u)\in [-C_z, C_z]\times U}\{ |\Delta_l^k(z,u,{\mathbold y}^\infty_t)|\le \epsilon\} )=1.
\end{align*}
\end{lemma}
\begin{proof}
We establish the first limit, and may deal with the second one in the same way.
Note that the event
\begin{align}
{\mathcal E}_{fC_z}^k\coloneqq \cap_{(z,u)\in [-C_z, C_z]\times U}\{ |\Delta_f^k(z,u,
{\mathbold y}^\infty_t  )|\le \epsilon\} \label{efcz}
\end{align}
is well defined since $\Delta_f^k$ is  continuous in $(z,u)$ and the intersection may be equivalently expressed  using only a countable number of values of $(z,u)$ in $[-C_z, C_z]\times U$.

Take any $\epsilon\in (0,1)$.
By (H2) and (H3), we can find $\delta_\epsilon>0$ such that
$|\Delta_f^k(z,u, {\mathbold y}^\infty_t )- \Delta_f^k(z',u',{\mathbold y}^\infty_t
  )| \le \epsilon/2$ whenever $|z-z'|+|u-u'|\le \delta_\epsilon$. For the selected $\delta_\epsilon$, we can find a fixed $p_0$ and $(z^j,u^j)\in [-C_z, C_z]\times U$, $j=1, \ldots, p_0$ such that for any $(z,u)\in [-C_z, C_z]\times U$, there exists some $j_0$ ensuring $|z-z^{j_0}|+|u-u^{j_0}|\le \delta_\epsilon$.

By Lemma \ref{lemma:DD2} and  Markov's inequality,  for any $\delta>0$,   there exists $K_{\delta, p_0 }$ such that for all $k\ge K_{ \delta, p_0}$,
 we have
\begin{align}
P(\{ |\Delta_f^k(z^j,u^j, {\mathbold y}^\infty_t  )|\le \epsilon/2\} )\ge 1-\delta/p_0, \quad \forall j,t.\label{p1del}
\end{align}
Let ${\mathcal E}^k_j $ denote the event $\{ |\Delta_f^k(z^j,u^j, {\mathbold y}^\infty_t)|\le \epsilon/2\}  $. By \eqref{p1del}, $P(\cap_{j=1}^{p_0}{\cal E}_j^k)\ge 1-\delta  $ for $k\ge  K_{\delta, p_0 } $.
Now if $\omega\in {\cal E}^k\coloneqq  \cap_{j=1}^{p_0}{\cal E}_j^k$, $k\ge K_{\delta, p_0 }  $, then for any  $(z,u)\in [-C_z, C_z]\times U$, we have
$
|\Delta_f^k(z,u,{\mathbold y}^\infty_t )|\le \epsilon.
$
Hence
$
{\cal E}^{k}\subset {\mathcal E}_{fC_z}^k .
$
It follows that for all $k\ge  K_{\delta, p_0 } $,
$
P( {\mathcal E}_{fC_z}^k
 )\ge 1-\delta.
$
Since $\delta\in (0, 1)$ is arbitrary and $ K_{\delta, p_0 }  $ does not depend on $t$, the first limit follows.
\end{proof}

\begin{lemma}\label{lemma:DDxuy}
We have
\begin{align*}
\lim_{k\to \infty}\sup_{t,u_{\cki}^N}E(|\Delta_f^k (x_{\cki}^N(t), u_{\cki}^N(t), {\mathbold y}^\infty_t)|+ |\Delta_l^k (x_{\cki}^N(t), u_{\cki}^N(t), {\mathbold y}^\infty_t)|)=0.
\end{align*}
\end{lemma}
\begin{proof}
Fix any $\epsilon \in (0, 1)$. By (H1) and (H2) we can find a sufficiently large $C_z$, independent of $(k,N)$, such that for all $u_{\cki}^N(\cdot)$,
\begin{align}
P\Big(\sup_{0\le t\le T} |x_{\cki}^N(t)| \le C_z\Big)\ge 1-\epsilon.\nonumber
\end{align}
Denote ${\cal E}_x= \{\sup_{0\le t\le T} |x_{\cki}^N(t)| \le C_z\}$.
By Lemma \ref{lemma:PDel}, for the above $\epsilon$ and ${\mathcal E}^k_{fC_z}$  given by \eqref{efcz}, there exists $K_0$ independent of $t$ such that for all $k\ge K_0$,
 $$
P({\mathcal E}^k_{fC_z} )\ge 1-\epsilon.
$$
Now if $\omega \in {\cal E}_x\cap {\mathcal E}^k_{fC_z} $, then
$|\Delta_f^k (x_{\cki}^N(t), u_{\cki}^N(t), {\mathbold y}^\infty_t)|\le \epsilon.$
We have
$P({\cal E}_x\cap {\mathcal E}^k_{fC_z})\ge 1-2\epsilon$, and so
$$
P(|\Delta_f^k (x_{\cki}^N(t), u_{\cki}^N(t), {\mathbold y}^\infty_t)|\le \epsilon)
\ge P({\cal E}_x\cap {\mathcal E}^k_{fC_z})  \ge 1-2\epsilon.
$$
It follows that for all $k\ge K_0$,
\begin{align}
E|\Delta_f^k (x_{\cki}^N(t), u_{\cki}^N(t), {\mathbold y}^\infty_t)|\le \epsilon + 2\epsilon C,\nonumber
\end{align}
where $C$ does not depend on $(u_{\cki}^N(\cdot), t)$.
The bound for $\Delta_l^k$ is similarly obtained.
\end{proof}

\begin{lemma}\label{lemma:edd}
We have
\begin{align}
\lim_{k\to \infty}\sup_{t,u_{\cki}^N(\cdot)}E(|\delta_f^k| +|\delta_l^k|)=0.\nonumber
\end{align}
\end{lemma}

\begin{proof} By  Lipschitz continuity of $(f,l)$,
we estimate
$E|\delta_f^k-\Delta_f^k(x_{\cki}^N, u_{\cki}^N,{\mathbold y}^\infty_t)|$ and
$E|\delta_l^k-\Delta_l^k(x_{\cki}^N, u_{\cki}^N,{\mathbold y}^\infty_t)|$,
and next apply Lemma \ref{lemma:e3} to show that they converge to zero as $k\to \infty$. Recalling Lemma   \ref{lemma:DDxuy}, we complete the proof.
\end{proof}

\begin{lemma}\label{lemma:edd0}
We have
\begin{align}
\lim_{k\to \infty}\sup_{t,u_{\cki}^N(\cdot)}E(|\delta_{f_0}^k| +|\delta_{l_0}^k|)=0.\nonumber
\end{align}
\end{lemma}
\begin{proof}
The proof is similar to that of Lemma \ref{lemma:edd} and the details are omitted.
\end{proof}

Denote
$$
\epsilon_{fl}^k=\sup_{t,u_{\cki}^N(\cdot)}E
(|\delta_{f_0}^k| +|\delta_{l_0}^k|+|\delta_f^k| +|\delta_l^k|).
$$

\begin{lemma}\label{lemma:jjstar}
For  any admissible control $u_{\cki}^N$ in System B and $J_{\cki}^*$
in \eqref{jistar},
\begin{align}
J_{\cki}^N(u_{\cki}^N) \ge \inf_{u_{\cki}} J_{\cki}^*(u_{\cki}) -C \epsilon_{fl}^k,\nonumber
\end{align}
where the constant $C$ does not depend on $u_{\cki}^N$.
\end{lemma}
\begin{proof}
Take any full state based Lipschitz feedback control $u_{\cki}^N$. It together with  the other agents's control laws generates the closed-loop state processes
$x_1^N(t), \ldots, x_N^N(t)$. Let $u_{\cki}^N(t,\omega)$ denote the realization as a non-anticipative process.  Now we take $\check u_{\cki}= u_{\cki}^N(t, \omega)$ in \eqref{yiofD} and let $\check x_{\cki}^\infty$  be the resulting state process.
It is clear from  \eqref{jistar} that
\begin{align}
J_{\cki}^*(\check u_{\cki}) \ge \inf_{u_{\cki}} J_{\cki}^*(u_{\cki}).\label{jjchecku}
\end{align}
Recalling \eqref{xindel} and  applying Gronwall's lemma  to estimate the difference $\check x_{\cki}^\infty- x_{\cki}^N$, we can show there exists $C$ independent of $u_{\cki}^N$ such that
$|J_{\cki}^N(u_{\cki}^N) -J_{\cki}^*(\check u_{\cki})|\le C \epsilon_{fl}^k,$
which combined with \eqref{jjchecku} completes the proof.
\end{proof}

\begin{lemma}\label{lemma:jjstarc} Let $\varphi_{I^*({\cki})} = \varphi(t,x, g_{I^*({\cki})})$ be the GMFG based control law \eqref{hatuIi}.
We have
\begin{align}
J^N_{\cki}(\varphi_{I^*({\cki})}) \le \inf_{u_{\cki}} J_{\cki}^*(u_{\cki}) + C\epsilon_{fl}^k.\nonumber
\end{align}
\end{lemma}
\begin{proof}
Let $\varphi_{I^*({\cki})}$ be applied to the two systems
\eqref{yiofD} and  \eqref{xindel}. We further use Gronwall's lemma to estimate  $E|x_{\cki}^\infty- x_{\cki}^N|$.
We obtain $ |J_{\cki}^N( \varphi_{I^*({\cki})} )-J_{\cki}^*(\varphi_{I^*({\cki})} )|\le  C\epsilon_{fl}^k$.
Note that $J_{\cki}^*( \varphi_{I^*({\cki})} )=\inf_{u_{\cki}} J_{\cki}^*(u_{\cki})$. This completes the proof.
\end{proof}

{\it Proof of Theorem \ref{theorem:eNash}.} It follows from Lemmas
 \ref{lemma:edd}, \ref{lemma:edd0}, \ref{lemma:jjstar} and \ref{lemma:jjstarc}. \qed

\section{The LQ Case}
\label{sec:lq}

This section considers a special class
of linear-quadratic-Gaussian (LQG) GMFG models. Consider the graph $G_k$ with vertices ${\cal V}_k=\{1, \ldots, M_k\}$  and graph adjacency matrix $g^k= [g^k_{jl}]$.
For  agent  ${\cal A}_i$ in subpopulation cluster ${\mathcal C}_q$ situated at node $q$, let the intra- and inter-cluster coupling terms  be denoted by $z_{0,i}$ and $z_i$, respectively, where
 $$z_{0,i}=\frac{1}{|\maC_q|}\sum_{j\in \maC_q}x_j,\quad
 z_i=\frac{1}{|M_k|}\sum_{l \in {\cal V}_k} g^k_{ql} \frac{1}{|\maC_l|}\sum_{j\in \maC_l}x_j, \quad x_j,\ z_{0,i},\ z_i \in {\mathbb R}^{n}.$$
The dynamics of  ${\cal A}_i$ are given by the linear system
\begin{align*}
dx_i= (Ax_i + D_0 z_{0,i}+  D z_i  +   B u_i)dt +\Sigma dw_i, \quad 1 \leq i \leq N, \label{eqn1}
\end{align*}
where
 $u_i \in {\mathbb R}^{n_u}$ is the control input,
$w_i \in {\mathbb R}^{n_w}$ is a standard Brownian motion, and $A$, $B$, $D_0$, $D$, $\Sigma$ are conformally dimensioned matrices. Assume $Ex_i(0)=x_0$ for all $i$.

The individual agent's cost function takes the form
\begin{align*}
J_i(u_i;\nu_i)= &{E}\int_0^T\big[(x_i-\nu_i)^T Q (x_i-\nu_i) +u_i^T R u_i\big]dt \\
& + E\big[(x_i(T)-\nu_i(T))^T Q_{T}(x_i(T)-\nu_i(T))\big],\quad 1 \leq i \leq N,
\end{align*}
where $Q$, $Q_T\geq 0$, $R>0$, and $\nu_i= \gamma_0 z_{0,i}+ \gamma z_i+\eta$ is the process tracked by  ${\cal A}_i$. Here $\eta\in \mathbb{R}^n$ and $\gamma_0, \gamma\in \mathbb{R}$.

In the infinite population and graphon limit case, denote the local mean  $ \int_{{\mathbb R}^n} x\mu_\alpha (dx)$ at  $t$ for an  $\alpha$-agent situated at vertex  $\alpha$  by $\bar x_\alpha$, and the graphon weighted mean
 $ \int_{0}^1  g(\alpha, \beta) \bar x_\beta d\beta $  by $z_{\alpha}$.
The  $\alpha$-agent's state equation is  given by
\begin{align*}
dx_\alpha= (Ax_\alpha +D_0 \bar x_\alpha +D z_\alpha  + B u_\alpha)dt +\Sigma dw_\alpha, \quad \alpha \in [0,1].
\end{align*}
The $\alpha$-agent's cost function is
\begin{align*}
J_\alpha(u_\alpha;\nu_\alpha)=& {E}\int_0^T\big[(x_\alpha-\nu_\alpha)^T Q (x_\alpha-\nu_\alpha) +u_\alpha^T R u_\alpha\big]dt \\
 &+E\big[ (x_\alpha(T)-\nu_\alpha(T))^T Q_{T}(x_\alpha(T)-\nu_\alpha(T))\big],
\end{align*}
where  $\nu_\alpha =\gamma_0 \bar x_\alpha+ \gamma z_\alpha+\eta$.

Consider the Riccati equation
\begin{align}
&0=\dot{\Pi}_t + A^T \Pi_t + \Pi_t A    -  \Pi_t B R^{-1}B^T \Pi_t +
  Q , \nonumber
  \end{align}
where $\Pi_T =   Q_T$, and
  \begin{align}
& 0=\dot{s}_\alpha(t) + (A-B R^{-1}B^T\Pi_t )^T {s}_\alpha(t) + {\Pi}_t (D_0 \bar x_{\alpha}(t)+ Dz_{\alpha}(t))  - Q\nu_\alpha(t), \nonumber
\end{align}
where   $s_\alpha(T) = -Q_T \nu_\alpha(T)$.
 The   best response for the $\alpha$-agent
  is given by
\begin{equation*}
                 u_\alpha(t) = - R^{-1}B^T[\Pi_t x_\alpha(t)+
                  s_\alpha(t)]  .
\end{equation*}
Now the mean state process of ${x}_{\alpha}$ is
   \begin{equation}
\dot{\bar x}_\alpha  = (A-BR^{-1}B^T\Pi_t+D_0)  \bar x_\alpha + D z_{\alpha} -B R^{-1}B{^T} s_{\alpha}, \quad \alpha \in [0,1].\nonumber
\end{equation}

The existence analysis reduces to verifying the existence and uniqueness of solutions for    the equation system
\begin{align}
&\dot{\bar x}_\alpha = (A-BR^{-1} B^T \Pi_t+D_0) \bar x_\alpha -BR^{-1} B^Ts_\alpha +D\int_0^1 g(\alpha, \beta) \bar x_\beta d\beta, \label{lqmfx}  \\
& \dot{s}_\alpha = - (A-BR^{-1} B^T \Pi_t)^T s_\alpha+ (\gamma_0 Q-\Pi_t D_0)\bar x_\alpha \label{lqmfs} \\
&\qquad +(\gamma Q- \Pi_t D) \int_0^1  g(\alpha, \beta) \bar x_\beta d\beta + Q\eta, \nonumber
\end{align}
where $\bar x_\alpha(0)=x_0$ and $s_\alpha(T)=- Q_T [\gamma_0 \bar x_\alpha(T) +\gamma \int_0^1 g(\alpha, \beta) \bar x_\beta (T) d\beta +\eta ] $.

To analyze  \eqref{lqmfx}--\eqref{lqmfs},
let $\Phi(t,s)$ and $\Psi(t,s)$  be the fundamental solution matrix of
\begin{align*}
\dot x = (A-BR^{-1} B^T \Pi_t+D_0) x, \qquad
\dot y = - (A-BR^{-1} B^T \Pi_t)^T y
\end{align*}
for $x(t), y(t)\in \mathbb{R}^n$. For the special case with $D_0=0$,
$ \Psi(t,s)= \Phi^T(s,t)$ holds.
We convert  the existence analysis into a fixed point problem.
We view  $\bar x_\beta(t)= \bar x(\beta,t)$ as a  function of  $(\beta,t)
$. Below we derive an equation for $\bar x_\alpha(t)$ by eliminating $s_\alpha(t)$.
Denote the function space
$D_\Lambda $ consisting of  continuous $\mathbb{R}^n$-valued functions on $[0,1]\times[0,T]$ with  norm $\|\check x\|=\sup_{\alpha, t}|\check x(\alpha,t)|$.
We use $|\cdot|$ to denote the Frobenius
norm of a vector or matrix. Define the operator $\Lambda$ as follows: for $\check x\in D_\Lambda$,
\begin{align*}
({\Lambda } \check x)( \alpha,t) =\
&\int_0^t \Phi(t, r) BR^{-1} B^T
\Big\{ \int_r^T \Psi(r, \tau) \Big[ (\gamma_0 Q-\Pi_\tau D_0)\check x(\alpha, \tau)\\
&\qquad\qquad+ (\gamma Q-\Pi_\tau D)\int_0^1 g(\alpha,\beta) \check x(\beta,\tau)d\beta  \Big]  d\tau\\
& + \Psi(r, T) Q_T\Big[ \gamma_0\check x(\alpha,T)+ \gamma \int_0^1 g(\alpha, \beta)\check x(\beta,T) d\beta \Big]  \Big\}dr\\
& + \int_0^t \Phi(t, r)  D \int_0^1 g(\alpha,\beta) \check x(\beta,r)d\beta dr.
\end{align*}
If (H5) holds, $\Lambda$ is from $D_\Lambda$ to itself.

The solution of the LQG GMFG reduces to finding a fixed point $\check x $ to the equation
\begin{align*}
\check x(\alpha, t) =& (\Lambda \check x)(\alpha,t)+ \Phi (t, 0)x_0\\
&+
 \int_0^t \Phi(t, r) BR^{-1} B^T
\Big[\! \int_r^T \Psi(r, \tau)    Q d\tau
 + \Psi(r, T) Q_T  \Big]\eta dr.
\end{align*}
Denote
$c_g= \max_\alpha\int_0^1 g(\alpha , \beta) d\beta$.
We have the bound for the operator norm:
\begin{align}
\|\Lambda \|\le c_\Lambda\coloneqq &  \sup_{t\in [0,T]}\Big\{\int_0^t \int_r^T | \Phi(t, r) BR^{-1}B^T \Psi(r, \tau)|\cdot (|\gamma_0 Q-\Pi_\tau D_0|\nonumber\\
&\qquad \qquad +  c_g|\gamma Q -\Pi_\tau  D |) d\tau dr\nonumber \\
&+   \int_0^t     \Big[ | \Phi(t, r) BR^{-1}B^T \Psi(r, T) Q_T|\cdot(|\gamma_0|+ c_g|\gamma| ) +c_g| \Phi(t, r) D|   \Big] dr\Big\}.\nonumber 
\end{align}
If $  c_\Lambda <1,$
$\Lambda$ is a contraction and \eqref{lqmfx}--\eqref{lqmfs} has  a unique solution.

As an example for illustration, we assume the graphon weighted mean at vertex $\alpha$  arises from an underlying {\it uniform attachment graphon}, and consequently
$$
z_{\alpha} =\int_{0}^1  (1-\max(\alpha, \beta)) \int_{\mathbb{R}^n}  x \mu_\beta(d x)  d\beta,\quad \alpha, \beta\in [0,1],
$$
where it is readily verified that the uniform attachment graphon
satisfies  (H5).



\section*{Appendix}
\renewcommand{\theequation}{A.\arabic{equation}}
\setcounter{equation}{0}
\renewcommand{\thetheorem}{A.\arabic{theorem}}
\setcounter{theorem}{0}

\begin{lemma}\label{lemma:wasscnt}
Assume \emph{(H1)--(H8)}. Let $\varphi_\alpha$ be the GMFG based best response \eqref{hatuIi} and $\mu_\alpha(t)$ the distribution of the closed-loop process $x_\alpha(t)$, $\alpha\in [0,1]$, in \eqref{clMV}
with initial distribution $\mu_0^x$. Then
we have
$$
\lim_{r\to 0}\sup_{|t-t^*|+|\beta-\beta^*|<r}W_1(\mu_\beta(t), \mu_{\beta^*}(t^*))=0,
$$
where $ t, t^*\in [0,T]$ and $\beta, \beta^*\in [0,1]$.
\end{lemma}

\begin{proof}
Step 1.
Take any $\beta , \beta^*\in [0,1]$. For  $\mu_G(\cdot)$ determined from the GMFG equations \eqref{MFGPDES} and \eqref{clMV},
define two processes
\begin{align*}
&dy_{\beta^*}=\widetilde f[y_{\beta^*}, \varphi(t, y_{\beta^*}, g_{\beta^*}), \mu_G; g_{\beta^*}]dt + \sigma dw_{\beta^*},\\
& dy_{\beta}=\widetilde f[y_{\beta}, \varphi(t, y_\beta, g_{\beta}), \mu_G; g_{\beta}]dt + \sigma dw_{\beta^*} ,
\end{align*}
where $y_{\beta^*}(0)= y_\beta(0)= x_i^N(0)$
and the same Brownian motion is used.
Then the distributions of $ y_{\beta^*}(t)$ and $y_{\beta}(t)$
are $\mu_{\beta^*} (t)$ and $\mu_\beta(t)$, respectively.
We obtain
\begin{align*}
&y_\beta(t)-y_{\beta^*}(t)\\
=&\int_0^t \Delta_{\beta,\beta^*}^0(s) ds+ \int_0^t\int_0^1\int_{\mathbb R} \Delta_{\beta, \beta^*}(s,z,\lambda) \mu_\lambda (s,dz) d\lambda ds,
\end{align*}
where
\begin{align*}
& \Delta_{\beta,\beta^*}^0(s)=  \int_{\mathbb R}f_0(y_{\beta}, \varphi(s, y_\beta, g_{\beta}),z)\mu_\beta(s,dz)-
\int_{\mathbb R} f_0(y_{\beta^*}, \varphi(s, y_{\beta^*}, g_{\beta^*}),z)\mu_{\beta^*}(s, dz),  \\
&\Delta_{\beta, \beta^*}(s,z,\lambda) =
f(y_{\beta}, \varphi(s, y_\beta, g_{\beta}),z )g(\beta, \lambda)\\
&\qquad\qquad\qquad - f(y_{\beta^*}, \varphi(s, y_{\beta^*}, g_{\beta^*}),z ) g(\beta^*, \lambda).
\end{align*}
We will simply write $\mu_\lambda(s, dz)$ as  $\mu_\lambda (dz)$ if the time argument is clear, where $\lambda$ is the vertex index.
Denote $\kappa_{\beta,\beta^*}(s)= |\varphi(s, y_{\beta^*}, g_{\beta}) - \varphi(s, y_{\beta^*}, g_{\beta^*}) |$, where the time argument $s$ in $y_\beta$ and  $y_{\beta^*}$  has been suppressed.
It follows that
\begin{align}
&|\Delta_{\beta,\beta^*}^0(s)|\le \nonumber \\
&\Big| \int_{\mathbb R}f_0(y_{\beta}, \varphi(s, y_\beta, g_{\beta}),z)\mu_\beta(s,dz)-
\int_{\mathbb R} f_0(y_{\beta}, \varphi(s, y_{\beta}, g_{\beta}),z)\mu_{\beta^*}(s, dz)\Big| \nonumber \\
& +\Big|\int_{\mathbb R}f_0(y_{\beta}, \varphi(s, y_\beta, g_{\beta}),z)\mu_{\beta^*}(s,dz)-
\int_{\mathbb R} f_0(y_{\beta^*}, \varphi(s, y_{\beta^*}, g_{\beta^*}),z)\mu_{\beta^*}(s, dz)\Big| \nonumber \\
&\le  C E|y_\beta-y_{\beta^*}| + C |y_\beta-y_{\beta^*}| + C| \varphi(s, y_\beta, g_{\beta}) - \varphi(s, y_{\beta^*}, g_{\beta^*})|  \nonumber \\
&\le  C E|y_\beta-y_{\beta^*}| + C_1 |y_\beta-y_{\beta^*}| +C \kappa_{\beta,\beta^*}(s),  \nonumber
\end{align}
where the second inequality is obtained using (H2), (H3), and the method in \eqref{mumyy}. The last inequality has used the uniform Lipschitz continuity of $\varphi_\beta$ in the space variable (see Lemma \ref{lemma:unilip}).
It follows that
\begin{align}
&E |\Delta_{\beta,\beta^*}^0(s)|
\le C_2 E|y_\beta(s)-y_{\beta^*}(s)|+ C
 E\kappa_{\beta,\beta^*}(s).\label{ef0del}
\end{align}

Next, we have
\begin{align}
&\Big|\int_0^1\int_{\mathbb R}\Delta_{\beta,\beta^*}(s,z,\lambda)\mu_\lambda (dz) d\lambda \Big|\nonumber\\
\le & \Big|\int_0^1\int_{\mathbb R}[f(y_{\beta}, \varphi(s, y_\beta, g_{\beta}),z )-f(y_{\beta^*}, \varphi(s, y_{\beta^*}, g_{\beta^*}),z )]g(\beta, \lambda) \mu_\lambda (dz) d\lambda \Big| \label{flip}  \\
&+ \Big|\int_0^1\int_{\mathbb R}f(y_{\beta^*},
\varphi(s, y_{\beta^*}, g_{\beta^*}),z ) [g(\beta, \lambda)- g(\beta^*, \lambda)]\mu_\lambda (dz) d\lambda\Big|\nonumber\\
=:&  I_f(s)+I_g(s).\nonumber
\end{align}
We have
\begin{align}
I_f(s)&\le \int_0^1 \int_{\mathbb R} C(|y_\beta -y_{\beta^*}| + \kappa_{\beta,\beta^*})g(\beta, \lambda) \mu_\lambda (dz) d\lambda \nonumber \\
&\le C(|y_\beta -y_{\beta^*}| + \kappa_{\beta,\beta^*})(s), \nonumber
\end{align}
where we have  used the Lipschitz property of $f$ and $\varphi_\beta$.
Therefore,
\begin{align}
EI_f(s)\le C(E|y_\beta(s) -y_{\beta^*}(s)| + E\kappa_{\beta,\beta^*}(s)).\label{eIfdel}
\end{align}

For any fixed value $y_{\beta^*}(s,\omega)$,
denote
\begin{align}
\xi_{\beta^*, s,\omega} (\lambda)= \int_{\mathbb R}f(y_{\beta^*}, \varphi(s, y_{\beta^*}, g_{\beta^*}),z )\mu_\lambda (dz).\nonumber
\end{align}
We have
\begin{align}
I_g(s) =\Big| \int_0^1 \xi_{\beta^*, s,\omega} (\lambda)g(\beta, \lambda) d\lambda - \int_0^1 \xi_{\beta^*, s,\omega} (\lambda)g(\beta^*, \lambda) d\lambda\Big|.\nonumber
\end{align}
Hence, by (H5),
$ I_g(s) \to 0$ $(\omega, s)$-a.e. as $\beta\to \beta^*$. It is clear $I_g(s)$ is bounded by a fixed constant since $f$ is a bounded function.
 For the fixed $\beta^*$, by Lemma \ref{lemma:contiBR}, the random variable $\kappa_{\beta,\beta^*}(s)$ is bounded and  converges to zero with probability one. Denote $\delta_g=\int_0^TEI_g(s)ds $ and
$ \delta_\kappa=\int_0^T E\kappa_{\beta,\beta^*}(s) ds$.
  By dominated convergence, we have
$$
\lim_{\beta\to \beta^*}( \delta_g +\delta_\kappa)= 0.
$$
By \eqref{ef0del}--\eqref{eIfdel},
it follows that
\begin{align}
E|y_\beta(t)-y_{\beta^*}(t)|\le C\int_0^t E|y_\beta(s)-y_{\beta^*}(s)|ds+ C( \delta_\kappa +  \delta_g).\nonumber
\end{align}
By Gronwall's lemma, we have
\begin{align}
\sup_{0\le t\le T}E|y_\beta(t)-y_{\beta^*}(t)|\le Ce^{CT}( \delta_\kappa+ \delta_g).\nonumber
\end{align}
Since
 $W_1(\mu_\beta (t), \mu_{\beta^*}(t))\le E|y_\beta(t)-y_{\beta^*}(t)|$,
then
\begin{align}
\sup_t W_1(\mu_\beta (t), \mu_{\beta^*}(t))\le C_{1}( \delta_\kappa+ \delta_g),
\label{w1kg}
\end{align}
where $\delta_\kappa$ and $\delta_g$ depend on $\beta^*$.

Step 2.
Now we consider given $(\beta^*, t^*)\in [0,1]\times [0,T]$. By use of the SDE of $y_\beta$ and elementary  estimates, we obtain
\begin{align}
\lim_{|t-t^*|\to 0}\sup_\beta W_1(\mu_{\beta}(t^*), \mu_{\beta}(t)) =0.\label{w1tst}
\end{align}
We have
\begin{align}
W_1(\mu_\beta (t), \mu_{\beta^*}(t^*))\le W_1(\mu_\beta (t), \mu_{\beta}(t^*)) + W_1(\mu_\beta (t^*), \mu_{\beta^*}(t^*)).\nonumber
\end{align}
Given any $\epsilon>0$, by \eqref{w1kg} and \eqref{w1tst} there exists $\delta_{\epsilon,\beta^*}>0$ such that whenever $|t-t^*|+|\beta-\beta^*|\le \delta_{\epsilon,\beta^*}$, we have
$$
W_1(\mu_\beta (t), \mu_{\beta}(t^*))\le \frac{\epsilon}{2}, \qquad W_1(\mu_\beta (t^*), \mu_{\beta^*}(t^*))\le \frac{\epsilon}{2}.
$$
Therefore, $W_1(\mu_\beta (t), \mu_{\beta^*}(t^*))\le \epsilon$. We conclude that $\mu_\beta(t)$ as a mapping from the compact space $[0, 1]\times [0,T]$ to ${\mathcal P}_1({\mathbb R})$ with the  metric $W_1(\cdot, \cdot)$ is continuous and hence must be uniformly continuous.
The lemma follows.
\end{proof}

\begin{lemma}\label{lemma:ST1}
Suppose the graphon $g$ satisfies \emph{(H5)} and \emph{(H11)}. Then for any
given measurable sets $ \mathcal{S}, \mathcal{ T} \subset [0,1]$, under \emph{(H9)} we have
\begin{align}
\lim_{k\to \infty}\Big|\int_{{\mathcal S}\times {\mathcal T}} ( g^k-g )dxdy\Big|=0.
\end{align}
\end{lemma}

\begin{proof}
 Step 1. We approximate ${\mathcal S},{\mathcal T}$ by open sets.
Let $\mu_{\rm L}$ denote the Lebesgue measure on $\mathbb{R}^d$, where the dimension $d$ will be clear from the context.
Consider the given sets
${\mathcal S},{\mathcal T}$, and choose an arbitrary $\epsilon>0$. Note that for any measurable set $A_1\subset\mathbb{R}^d$ and any $\delta_0>0$, there exists an open set $A_2\supset A_1 $ such that $\mu_{\rm L}(A_2\backslash A_1)\le \delta_0$ (see e.g. \cite{N83}). So there exist open sets $\mathcal{S}^o\subset \mathbb{R}$ and ${\mathcal T}^o\subset \mathbb{R}$ such that
$\mathcal{S} \subset \mathcal{S}^o$, ${\mathcal  T}\subset{\mathcal  T}^o$ and
$\mu_{\rm L}(\mathcal{S}^o\backslash\mathcal{S}) \le \epsilon$, $\mu_{\rm L}({\mathcal T}^o\backslash {\mathcal T}) \le \epsilon$.

Define the new open sets $\mathcal{S}_1^o=\mathcal{S}^o\cap (0,1)$ and ${\mathcal T}^o_1={\cal T}^o\cap (0,1)$.
Each open set in $\mathbb{R}$ may be written as the union of at most countable disjoint open intervals \cite{N83}; among such a union for $\mathcal{S}_1^o $, we may find a finite integer $s^*$  (depending on $({\mathcal S}, \epsilon) $) and constituent disjoint open intervals $I_i^\mathcal{S}\subset [0,1]$, $1\le i\le s^*$, such that $U_{s^*}\coloneqq \cup_{i=1}^{s^*} I_i^{\mathcal S} \subset {\mathcal S}_1^o$ and  $\mu_{\rm L} ( \mathcal{S}_1^o\backslash U_{s^*})\le \epsilon$.  Similarly, we find a finite integer $t^*$ and disjoint open intervals $I_i^{\cal T}\subset [0,1]$ such that   $U_{t^*}\coloneqq \cup_{j=1}^{t^*} I_j^{\mathcal T} \subset {\mathcal T}_1^o $ and  $\mu_{\rm L} ( {\mathcal T}_1^o\backslash U_{t^*})\le \epsilon$. Here the choice of $(s^*, t^*)$ depends on $(\mathcal{S},\mathcal{T}, \epsilon)$.

By the construction of $U_{s^*} $ and $U_{t^*}$,
 we have the bound for the measure of the following  symmetric differences:
\begin{align}
\mu_{\rm L}(\mathcal{S}\Delta U_{s^*}) \le 2{\epsilon}, \quad
\mu_{\rm L}(\mathcal{T}\Delta U_{t^*}) \le 2{\epsilon}, \nonumber
\end{align}
which implies $\mu_{\rm L} (( \mathcal{S}\times \mathcal{T})\Delta (U_{s^*} \times U_{t^*}) )\le 6\epsilon$.
Since  $|g^k-g|\le 1$ for any $x,y$, we have
\begin{align}
\Big|\int_{\mathcal{S}\times \mathcal{T}} ( g^k-g )dxdy - \eta_k\Big|\le {6\epsilon}, \label{ST2Ust}
\end{align}
where
\begin{align}
\eta_k&\coloneqq \Big|\int_{U_{s^*}\times U_{t^*}} ( g^k-g )dxdy\Big|.\nonumber
\end{align}

Step 2. Blow we estimate $\eta_k$. Under (H9)
we take a sufficiently large $K_0$, depending on ${s^*}$ (and so on $({\mathcal S}, \epsilon)$), such that for all $k\ge K_0$, $$
\frac{s^*}{M_k}\le {\epsilon}.
$$
Consider $k\ge K_0$.
 We select from the subintervals
$I^{k}_1, \ldots, I_{M_k}^{k}$ of equal length $1/M_k$ in the partition of $[0,1]$  such that  a subinterval is selected whenever its interior is  contained in $U_{s^*}$. The method here is to fill $U_{s^*}$ as much as possible from inside by these subintervals.  This procedure determines  a subcollection denoted by $I^{k}_{i_r}$, $r=1, \ldots, r_k$. Denote $\hat U_{s^*} = \cup_{r=1}^{r_k} I^{k}_{i_r}$.
Then the interior of $\hat U_{s^*} $ is contained in $U_{s^*}$. We need to estimate the measure for the part of $U_{s^*}$ not covered by $\hat U_{s^*} $.
We check $I_i^{\mathcal S}$, $1\le i\le s^*$, to obtain two cases:
(i) $I_i^{\mathcal S}\subset \hat U_{s^*} $, (ii) $I_i^{\mathcal S}$ has a portion (allowed to be equal to its whole) of positive measure staying outside  $\hat U_{s^*}$. For case (ii),  the portion of $I_i^{\mathcal S}$ that is not covered by $\hat U_{s^*}$ consists of either one interval, as part or the whole of $I_i^{\mathcal S}$, or  two intervals each having an endpoint of
$I_i^{\mathcal S}$ as its boundary; hence the measure of that portion is less than $2/M_k$.   It follows that
\begin{align}
\mu_{\rm L}(U_{s^*} \backslash \hat U_{s^*}  )\le \frac{2s^*}{M_k}
\le 2{\epsilon}. \label{muuue}
\end{align}

By \eqref{muuue}, for all  $k\ge K_0$, we have
\begin{align}
\Big|\int_{U_{s^*}\times U_{t^*}} ( g^k-g )dxdy-\int_{\hat U_{s^*}\times U_{t^*}} ( g^k-g )dxdy\Big|  \le {2\epsilon}.\label{s2shat}
\end{align}

Step 3. Now for $k\ge K_0$ we check
\begin{align}
\hat \eta_k\coloneqq \Big|\int_{\hat U_{s^*}\times U_{t^*}} ( g^k-g )dxdy\Big|.\nonumber
\end{align}
By (H5), for the selected $U_{t^*}$, $\int_{U_{t^*}} g(x,y) dy $ as a function of $x$ is uniformly continuous on $[0,1]$.
 So for $\epsilon$ chosen in Step 1,  there exists $\delta>0$ (depending on $g$, $\epsilon$ and $U_{t^*}$) such that
\begin{align}
\Big|\int_{U_{t^*}} g(x,y) dy -\int_{U_{t^*}} g(x',y)dy  \Big|\le {\epsilon}
\label{gxxp}
\end{align}
whenever $|x-x'|\le \delta$. For the above $\delta$, we fix  $K_1\ge K_0$ such that for all $k\ge K_1$, we have
$1/{M_k}\le 2\delta$.
Note that  we use $(I_{i_r}^k)^*$ to denote the midpoint of the interval $I_{i_r}^k$.
Now for $k\ge K_1$, we have
\begin{align*}
\hat \eta_k 
&= \Big|\sum_{r=1}^{r_k} \int_{I^{k}_{i_r}} \int_{U_{t^*}} [g^k(x,y)-g(x,y)] dydx\Big|\\
&\le\Big|  \sum_{r=1}^{r_k} \int_{I^{k}_{i_r}} \int_{U_{t^*}} [g^k((I^{k}_{i_r})^*,y )- g((I^{k}_{i_r})^*,y )]dy dx\Big| +{\epsilon}\\
& =  \Big|\sum_{r=1}^{r_k} \frac{1}{M_k }\int_{U_{t^*}} [g^k((I^{k}_{i_r})^*,y )- g((I^{k}_{i_r})^*,y )]dy\Big| +{\epsilon}\\
&\le \frac{1}{M_k }\sum_{r=1}^{r_k}\zeta_k +{\epsilon},
\end{align*}
where
\begin{align}
\zeta_k\coloneqq \Big|\int_{U_{t^*}} [g^k((I^{k}_{i_r})^*,y )- g((I^{k}_{i_r})^*,y )]dy\Big|.\nonumber
\end{align}
The first inequality follows from \eqref{gxxp} and  $\mu_{\rm L} (\cup_{r=1}^{r_k} I^{k}_{i_r})\le 1$.

Step 4. Now we estimate
$\zeta_k$.
As in Step 2, we  take a sufficiently large $K_2\ge K_1$, depending on $ ({t^*, \epsilon})$, such that for all $k\ge K_2$,
${t^*}/{M_k}\le {\epsilon}.$
For $k\ge K_2$ and  the subintervals
$I^{k}_1, \ldots, I_{M_k}^{k}$, as in Step 2, we  select a subcollection denoted by $I^{k}_{j_\tau}$, $\tau=1, \ldots, \tau_k$, each of which is selected whenever its interior is  contained in $U_{t^*}$. Then it follows that
\begin{align}
\mu_{\rm L}(U_{t^*} \backslash  \cup_{\tau=1}^{\tau_k} I^{k}_{j_\tau} )\le \frac{2t^*}{M_k}\le 2\epsilon . \label{muLtau}
\end{align}
By \eqref{muLtau}, we have for all $k\ge K_2$,
\begin{align*}
 \zeta_k
\le &\Big|\int_{\cup_{\tau=1}^{\tau_k} I^{k}_{j_\tau} } [g^k((I^{k}_{i_r})^*,y )- g((I^{k}_{i_r})^*,y )]dy\Big|+2{\epsilon}\\
\le & \sum_{\tau=1}^{\tau_k}\Big|\frac{ g^k_{i_rj_\tau} }{M_k}- \int_{\beta\in I^{k}_{j_\tau} } g_{(I^{k}_{i_r})^* ,\beta}
d\beta\Big| +2{\epsilon}.
\end{align*}
We write $g(\alpha, \beta)$ as $g_{\alpha, \beta}$.

Step 5. Note that $r_k, \tau_k\le M_k$.   Subsequently, by Step 3 and Step 4, we have for $k\ge K_2$,
\begin{align}
\hat \eta_k &\le  \frac{1}{M_k }\sum_{r=1}^{r_k}\Big[ \sum_{\tau=1}^{\tau_k}\Big|\frac{ g^k_{i_rj_\tau} }{M_k}- \int_{\beta\in I^{k}_{j_\tau} } g_{(I^{k}_{i_r})^* ,\beta}
d\beta\Big| +2{\epsilon}  \Big]+ {\epsilon} \nonumber\\
 & \le \frac{1}{M_k }\sum_{r=1}^{r_k} \sum_{\tau=1}^{\tau_k}\Big|\frac{ g^k_{i_rj_\tau} }{M_k}- \int_{\beta\in I^{k}_{j_\tau} } g_{(I^{k}_{i_r})^* ,\beta}
d\beta\Big|  + {3\epsilon}\nonumber\\
&\le \max_i \sum_{j=1}^{M_k} \Big|\frac{ g^k_{{\mathcal C}_i{\mathcal C}_j} }{M_k}- \int_{\beta\in I_j^{k} } g_{(I^{k}_{i})^* ,\beta}
d\beta\Big|  + {3\epsilon}. \label{hatetamax}
\end{align}
 By \eqref{ST2Ust}, \eqref{s2shat} and \eqref{hatetamax}, we obtain for all $k\ge K_2$ depending on $({\mathcal S},{\mathcal T},\epsilon)$,
\begin{align*}
\Big|\int_{\mathcal{S}\times \mathcal{T}} ( g^k-g )dxdy\Big|
 &\le \max_i \sum_{j=1}^{M_k} \Big|\frac{ g^k_{{\mathcal C}_i{\mathcal C}_j} }{M_k}- \int_{\beta\in I_j^{k} } g_{(I^{k}_j)^* ,\beta}
d\beta\Big| +11\epsilon.
\end{align*}
The lemma follows.
\end{proof}

\bibliographystyle{amsplain}
\bibliography{caineshuangGmfgRefV2}

\providecommand{\bysame}{\leavevmode\hbox to3em{\hrulefill}\thinspace}
\providecommand{\MR}{\relax\ifhmode\unskip\space\fi MR }
\providecommand{\MRhref}[2]{%
  \href{http://www.ams.org/mathscinet-getitem?mr=#1}{#2}
}
\providecommand{\href}[2]{#2}
\begin{thebibliography}{10}

\bibitem{BCW20}
Erhan Bayraktar, Suman Chakraborty, and Ruoyu Wu, \emph{Graphon mean field
  systems}, arXiv:2003.13180 (2020).

\bibitem{borgs2006graph}
Christian Borgs, Jennifer Chayes, L{\'a}szl{\'o} Lov{\'a}sz, Vera~T S{\'o}s,
  Bal{\'a}zs Szegedy, and Katalin Vesztergombi, \emph{Graph limits and
  parameter testing}, Proc. the thirty-eighth annual ACM symposium on Theory of
  computing, 2006, pp.~261--270.

\bibitem{borgs2008convergent}
Christian Borgs, Jennifer~T Chayes, L{\'a}szl{\'o} Lov{\'a}sz, Vera~T S{\'o}s,
  and Katalin Vesztergombi, \emph{Convergent sequences of dense graphs {I}:
  Subgraph frequencies, metric properties and testing}, Advances in Mathematics
  \textbf{219} (2008), no.~6, 1801--1851.

\bibitem{borgs2012convergent}
\bysame, \emph{Convergent sequences of dense graphs {II}. multiway cuts and
  statistical physics}, Annals of Mathematics \textbf{176} (2012), no.~1,
  151--219.

\bibitem{caines2015mean}
Peter~E Caines, \emph{Mean field games}, Encyclopedia of Systems and Control
  (2015), 706--712.

\bibitem{CainesHuangCDC2018}
Peter~E Caines and Minyi Huang, \emph{Graphon mean field games and the {GMFG}
  equations}, Proc. 57th IEEE CDC (Miami Beach, FL, USA), 2018, pp.~4129--4134.

\bibitem{CH19}
\bysame, \emph{Graphon mean field games and the {GMFG} equations:
  $\epsilon$-{N}ash equilibria}, Proc. the 58th IEEE CDC (Nice, France), 2019,
  pp.~286--292.

\bibitem{CHM17}
Peter~E Caines, Minyi Huang, and Roland~P Malham{\'e}, \emph{Mean field games},
  Handbook of Dynamic Game Theory (Tamer Ba\c{s}ar and Georges Zaccour, eds.),
  Springer, Berlin, 2017, pp.~345--372.

\bibitem{CarmonaDelarueBook2018a}
Rene Carmona and Francois Delarue, \emph{Probabilistic theory of mean field
  games with applications {I}}, vol.~83, Springer International Publishing,
  2018.

\bibitem{CarmonaDelarueBook2018b}
\bysame, \emph{Probabilistic theory of mean field games with applications
  {II}}, vol.~84, Springer International Publishing, 2018.

\bibitem{delarue2017mean}
Fran{\c{c}}ois Delarue, \emph{Mean field games: A toy model on an
  {E}rd{\"o}s-{R}enyi graph}, ESAIM: Proceedings and Surveys \textbf{60}
  (2017), 1--26.

\bibitem{doob1953stochastic}
Joseph~L Doob, \emph{Stochastic processes}, Wiley, New York, 1953.

\bibitem{FR75}
Wendell~H. Fleming and Raymond~W. Rishel, \emph{Deterministic and stochastic
  optimal control}, Springer-Verlag, New York, 1975.

\bibitem{gallagher2019newtonnavierstokes}
Isabelle Gallagher, \emph{From {N}ewton to {N}avier-{S}tokes, or how to connect
  fluid mechanics equations from microscopic to macroscopic scales}, Bulletin
  of the American Math. Society \textbf{56} (2013), no.~1, 65--85.

\bibitem{gallagher2013newton}
Isabelle Gallagher, Laure Saint-Raymond, and Benjamin Texier, \emph{From
  {N}ewton to {B}oltzmann: hard spheres and short-range potentials}, European
  Mathematical Society, 2013.

\bibitem{ShuangPeterCDC17}
Shuang Gao and Peter~E. Caines, \emph{The control of arbitrary size networks of
  linear systems via graphon limits: An initial investigation}, Proc. 56th IEEE
  CDC (Melbourne, Australia), December 2017, pp.~1052--1057.

\bibitem{ShuangPeterNetSci17}
\bysame, \emph{Controlling complex networks of linear systems via graphon
  limits}, \textup{Presented at} the Symposium of Controlling Complex Networks
  of NetSci17, \textup{Indianapolis, IN, USA} (2017).

\bibitem{ShuangPeterSIAM17}
\bysame, \emph{Minimum energy control of arbitrary size networks of linear
  systems via graphon limits}, \textup{Presented at} the SIAM Workshop on
  Network Science, \textup{Pittsburgh, PA, USA} (2017).

\bibitem{ShuangPeterCDC18}
\bysame, \emph{Graphon linear quadratic regulation of large-scale networks of
  linear systems}, Proc. 57th IEEE Conference on Decision and Control (Miami
  Beach, FL, USA), December 2018, pp.~5892--5897.

\bibitem{GC20}
\bysame, \emph{Graphon control of large-scale networks of linear systems}, IEEE
  Transactions on Automatic Control \textbf{65} (2020), no.~10, 4090--4105.

\bibitem{G15}
Olivier Gu{\'e}ant, \emph{Existence and uniqueness result for mean field games
  with congestion effect on graphs}, Applied Mathematics \& Optimization
  \textbf{72} (2015), no.~2, 291--303.

\bibitem{herron2008partial}
Isom~H Herron and Michael~R Foster, \emph{Partial differential equations in
  fluid dynamics}, Cambridge University Press, 2008.

\bibitem{HCM07}
Minyi Huang, Peter~E Caines, and Roland~P Malham{\'e}, \emph{Large-population
  cost-coupled {LQG} problems with nonuniform agents: individual-mass behavior
  and decentralized $\varepsilon$-{N}ash equilibria}, IEEE Transactions on
  Automatic Control \textbf{52} (2007), no.~9, 1560--1571.

\bibitem{HCM10}
\bysame, \emph{The {NCE} (mean field) principle with locality dependent cost
  interactions}, IEEE Transactions on Automatic Control \textbf{55} (2010),
  no.~12, 2799--2805.

\bibitem{HMC06}
Minyi Huang, Roland~P Malham{\'e}, and Peter~E Caines, \emph{Large population
  stochastic dynamic games: closed-loop {M}c{K}ean-{V}lasov systems and the
  {N}ash certainty equivalence principle}, Communications in Information \&
  Systems \textbf{6} (2006), no.~3, 221--252.

\bibitem{kaliuzhnyi2017semilinear}
Dmitry Kaliuzhnyi-Verbovetskyi and Georgi~S Medvedev, \emph{The semilinear heat
  equation on sparse random graphs}, SIAM Journal on Mathematical Analysis
  \textbf{49} (2017), no.~2, 1333--1355.

\bibitem{karatzas2012brownian}
Ioannis Karatzas and Steven Shreve, \emph{Brownian motion and stochastic
  calculus}, vol. 113, Springer Science \& Business Media, 2012.

\bibitem{KY15}
Vassili Kolokoltsov and Wei Yang, \emph{Sensitivity analysis for {HJB}
  equations with an application to coupled backward-forward systems}, arXiv
  preprint arXiv:1303.6234v2 (2015).

\bibitem{Ladyzh68}
Olga~Aleksandrovna Ladyzhenskaya, NN~Ural'ceva, and VA~Solonnikov, \emph{Linear
  and quasi-linear equations of parabolic type}, American Mathematical Society,
  1968.

\bibitem{LL06a}
Jean-Michel Lasry and Pierre-Louis Lions, \emph{Jeux \'a champ moyen. {I} - le
  cas stationnaire}, Comptes Rendus Math{\'e}matique \textbf{343} (2006),
  no.~9, 619--625.

\bibitem{LL06b}
\bysame, \emph{Jeux \'a champ moyen. {II} horizon fini et controle optimal},
  Comptes Rendus Math{\'e}matique \textbf{343} (2006), no.~10, 679--684.

\bibitem{lovasz2012large}
L{\'a}szl{\'o} Lov{\'a}sz, \emph{Large networks and graph limits}, vol.~60,
  American Mathematical Soc., 2012.

\bibitem{lovasz2006limits}
L{\'a}szl{\'o} Lov{\'a}sz and Bal{\'a}zs Szegedy, \emph{Limits of dense graph
  sequences}, Journal of Combinatorial Theory, Series B \textbf{96} (2006),
  no.~6, 933--957.

\bibitem{medvedev2014nonlineardense}
Georgi~S Medvedev, \emph{The nonlinear heat equation on dense graphs and graph
  limits}, SIAM J. Math. Anal. \textbf{46} (2014), no.~4, 2743--2766.

\bibitem{medvedev2014nonlinear}
\bysame, \emph{The nonlinear heat equation on w-random graphs}, Archive for
  Rational Mechanics and Analysis \textbf{212} (2014), no.~3, 781--803.

\bibitem{N83}
I.~P. Natanson, \emph{Theory of functions of a real variable, vol. {I}}, F.
  Ungar Publishing Co., 1983, 5th printing.

\bibitem{PariseOzdagler2018}
Francesca Parise and Asuman Ozdaglar, \emph{Graphon games}, arXiv preprint
  arXiv:1802.00080 (2018).

\bibitem{pauli2000thermodynamics}
Wolfgang Pauli and Charles~P Enz, \emph{Thermodynamics and the kinetic theory
  of gases}, vol.~3, Courier Corporation, 2000.

\bibitem{QT15}
Cristobal Quininao and Jonathan Touboul, \emph{Limits and dynamics of randomly
  connected neuronal networks}, Acta Appl Math \textbf{136} (2015), 167--192.

\bibitem{SEC16}
Nevroz Sen and Peter~E Caines, \emph{Mean field game theory with a partially
  observed major agent}, SIAM Journal on Control and Optimization \textbf{54}
  (2016), no.~6, 3174--3224.

\bibitem{szn91}
Alain-Sol Sznitman, \emph{Topics in propagation of chaos}, Ecole d'{\'e}t{\'e}
  de probabilit{\'e}s de Saint-Flour XIX—1989, Springer, 1991, pp.~165--251.

\end{thebibliography}


\end{document}